\documentclass[10pt,a4paper]{amsart}
\usepackage{amsmath, amssymb,amsthm}
\usepackage{color} 
\usepackage{hyperref}
\usepackage[all]{xy}
\usepackage{bbm}

\usepackage[frenchb, main=english]{babel}
\usepackage{enumerate, enumitem}
\usepackage{comment}
\usepackage{mathrsfs}
\usepackage{MnSymbol} 
\usepackage{tikz} 
\usepackage{tabularx}
\usepackage[T1]{fontenc}
\usepackage[utf8]{inputenc}

\usepackage{comment}

\newtheorem{theoremletter}{Theorem}

\newtheorem{corollaryletter}{Corollary}

\newtheorem{thm}{Theorem}[section]
\newtheorem{prop}[thm]{Proposition }
\newtheorem{lemma}[thm]{Lemma}
\newtheorem{cor}[thm]{Corollary}
\newtheorem{defn}[thm]{Definition}

\newtheorem*{thm-num}{Theorem} 

\theoremstyle{remark}
\newtheorem{rem}{Remark}[section]

\newcommand{\C}{\mathbb{C}}
\newcommand{\Q}{\mathbb{Q}}

\newcommand{\Z}{\mathbb{Z}}

\newcommand{\Qbar}{\overline \Q }

\newcommand{\cC}{{\mathcal{C}}}
\newcommand{\cE}{{\mathcal{E}}}
\newcommand{\cD}{{\mathcal{D}}}

\newcommand{\gP}{\mathfrak{P}}
\newcommand{\cO}{\mathcal{O}}

\newcommand{\cR}{\mathcal{R}}

\newcommand{\cT}{\mathcal{T}}
\newcommand{\cW}{\mathcal{W}}
\newcommand{\gm}{\mathfrak{m}}
\newcommand{\gC}{\mathfrak{C}}
\newcommand{\gc}{\mathfrak{c}}
\newcommand{\gn}{\mathfrak{n}}
\newcommand{\gp}{\mathfrak{p}}
\newcommand{\gq}{\mathfrak{q}}
\newcommand{\full}{\mathrm{full}}

\DeclareMathOperator{\ad}{ad} 
 
\DeclareMathOperator{\End}{End}
\DeclareMathOperator{\Frob}{Frob}
\DeclareMathOperator{\Ind}{Ind}
\DeclareMathOperator{\Gal}{Gal}
\DeclareMathOperator{\GL}{GL}
\DeclareMathOperator{\rH}{H}

\DeclareMathOperator{\Hom}{Hom}
\DeclareMathOperator{\rk}{rk}

\DeclareMathOperator{\Tr}{Tr}

\DeclareMathOperator{\Res}{Res}
\DeclareMathOperator{\PGL}{PGL}

\usepackage{mathrsfs}
\usepackage{MnSymbol} 
\usepackage{tikz} 
\usepackage{tabularx}
\begin{document}
\title[On the Hilbert eigenvariety at exotic and CM weight 1 points]{On the Hilbert eigenvariety at exotic and CM\\ classical weight 1 points}
\author{Adel Betina}
\address{The University of Vienna}
\email{adelbetina@gmail.com}

\author{Shaunak V. Deo}
\address{School of Mathematics, Tata Institute of Fundamental Research, Homi Bhabha Road, Mumbai 400005, India}
\email{deoshaunak@gmail.com}

\author{Francesc Fit\'e}
\address{Department of Mathematics\\
Massachussetts Institute of Technology \\
77 Massachussetts Avenue \\ 02139 Cambridge, Massachussetts\\ USA}
\email{ffite@mit.edu}
\urladdr{http://www-math.mit.edu/~ffite/}


\begin{abstract}
Let $F$ be a totally real number field and let $f$ be a classical cuspidal $p$-regular Hilbert modular eigenform over $F$ of parallel weight~$1$. Let $x$ be the point on the $p$-adic Hilbert eigenvariety $\cE$ corresponding to {an ordinary} $p$-stabilization of $f$. We show that if the $p$-adic Schanuel Conjecture is true, then $\cE$ is smooth at $x$ if $f$ has CM. If we additionally assume that $F/\Q$ is Galois, we show that the weight map is \'etale at $x$ if $f$ has either CM or exotic projective image {(which is the case for almost all cuspidal Hilbert modular eigenforms of parallel weight $1$)}. We prove these results by showing that the completed local ring of the eigenvariety at $x$ is isomorphic to a universal nearly ordinary Galois deformation ring.  %As an application, we show that our results yield the non-vanishing of certain anti-cyclotomic Katz $p$-adic $L$-function at $s=0$.
\end{abstract}

\maketitle

\section{Introduction}\label{section: introduction}

{The main goal of this paper is to study the geometry of the eigenvariety of Hilbert modular forms at classical points of parallel weight one. Before proceeding further, we will first fix some notations and describe the objects of interest.}
Let $p$ be a prime number, $F \subseteq \Qbar$ be a totally real number field of degree~$d$ over $\mathbb{Q}$ with ring of integers $\cO_F$, and~$\gn$ be an ideal of $\cO_F$ coprime to $p$. 

We will denote by $\cE$ the $p$-adic Hilbert eigenvariety of tame level~$\mathfrak{n}$ constructed by Andreatta, Iovita and Pilloni in \cite{AIP16}, parameterizing  systems of Hecke eigenvalues of overconvergent cuspidal Hilbert modular eigenforms over $F$ of tame level $\gn$, {having} weights of same parity and finite slope.  Recall that there exists a locally finite morphism $w=(k,v): \cE \rightarrow \cW$ called the weight map, where $\mathcal{W}$ is the rigid space over~$\Q_{p}$ representing morphisms $\Z_p^\times\times (\cO_F\otimes \Z_p)^\times \rightarrow \mathbb{G}_{m}$. Recall also that locally on $\cE$ and~$\cW$, the morphism $w$ is finite, open and surjective, though it is not necessarily flat. Thus, the $p$-adic eigenvariety~$\cE$ is equidimensional of dimension $d+1$.

Let $f$ be a classical cuspidal Hilbert modular eigenform over $F$ of tame level $\gn$ having weights of same parity. A $p$-stabilization of $f$ with finite slope (when it exists) defines a point $x$ on $\cE$. A well-known result of Hida (\cite{Hid86} if $F=\Q$, \cite{Hida89} in general) asserts that $w$ is \'etale at {$p$-ordinary eigenforms of cohomological weights (i.e. all of its weights have the same parity and are at least $2$)}. The classicality criterion of Bijakowski and Pilloni-Stroh for overconvergent forms implies that $w$ is \'etale at $x$ if $f$ is $p$-regular and has a non-critical slope (see \cite{Bijakowski} and \cite[Thm.~1.1]{PS12}). However, their results do not apply in the parallel weight $1$ case.

From now on, we assume that $f$ is a classical cuspidal Hilbert modular eigenform over $F$ of parallel weight $1$, tame level~$\mathfrak{n}$ and nebentypus $\chi_f$, and let $x$ be the point of $\cE$ corresponding to a $p$-stabilization of $f$ with finite slope. Our aim is to study the geometry of $\cE$ at $x$. 

Before recalling the previous results (in both the elliptic modular and Hilbert modular cases) and stating the results of this paper, we give some definitions first.

Fix an embedding $\iota_p: \overline \Q\hookrightarrow \overline\Q_p$ and let $G_F$ denote the absolute Galois group of $F$. {Let $\gn'$ be the level of the newform underlying $f$.} To introduce the definition of being  \og $p$-regular \fg, recall that there exists a totally odd Artin representation $
\varrho_f : G_F\rightarrow \GL_2(\overline \Q_p)
$
such that for all primes $\gq\nmid \gn' $, the representation $\varrho_f$ is unramified at $\gq$ and
$$
\Tr\varrho_f(\Frob_\gq)=a(\gq,f), \qquad \det\varrho_f(\Frob_\gq)=\chi_f(\Frob_\gq)\,,
$$
where $\Frob_\gq$ denotes an arithmetic Frobenius at $\gq$ and $a(\gq,f)$ is the $T_\gq$-eigenvalue of $f$ (see \cite{R-T}, \cite{Car86}, \cite{Wil88}, and \cite{Tay89}).

 For any prime $\gp|p$ of $\cO_F$ {not dividing $\gn'$,} let $\alpha_\gp$ and $\beta_\gp$ be the roots of the polynomial 
$$
X^2- \Tr \varrho_f(\Frob_{\gp}) X+\chi_f(\Frob_{\gp})\,.
$$
For any prime $\gp|p$ of $\cO_F$ {dividing $\gn'$,} let $\alpha_\gp$ and $\beta_\gp$ be the roots of the polynomial 
$X^2- a(\gp,f) X$, where $a(\gp,f)$ is the $U_\gp$-eigenvalue of the newform underlying $f$.
%We say that $\varrho_f$ is $p$-regular 
We say that $f$ (or $x$) is $p$-regular if $\alpha_\gp\neq\beta_\gp$ for all $\gp \mid p$ of $\cO_F$.
In all of our main results, we assume that $f$ is $p$-regular.
%$\varrho_f$ is $p$-regular.

%\begin{rem}

%Note that there exist classical cuspidal Hilbert modular forms of parallel weight $1$ whose Galois representation is unramified at some but not all primes of $F$ above $p$.
%For example, let $K$ be a CM field and $F$ be the corresponding totally real field. Suppose there are at least $2$ primes of $F$ lying above $p$ which are split in $K$. Let $\gp_1$ and $\gp_2$ be two such distinct primes of $F$ lying above $p$.
%Let $\psi : G_K \to \Qbar_p^{\times}$ be a character of finite order such that the conductor of $\psi$ is not co-prime to $\gp_1$ but it is co-prime to $\gp_2$.
%So  the representation $\Ind_K^F \psi$ of $G_F$ is ramified at $\gp_1$ but unramified at $\gp_2$ and there exists a classical cuspidal Hilbert modular form $f$ of parallel weight $1$ such that $\varrho_f \simeq \Ind_K^F \psi$ (see \cite[\S 2.1]{Bet} for more details).
%\end{rem}

Observe that, in the notation established above, the $U_\gp$-eigenvalue of a $p$-stabilization of $f$ is either $\alpha_\gp$ or $\beta_\gp$ if $\gp \nmid \gn'$ and it is $a(\gp,f)$ if $\gp \mid \gn'$. 
{In the latter case, if $a(\gp,f) \neq 0$, then, combining the local-global compatibility and the fact that image of $\varrho_f$ is finite, we see that $a(\gp,f)$ is an $n$-th root of unity.}
Therefore, it follows that a $p$-stabilization of $f$ with finite slope is always ordinary. So for the rest of the article, we will use this convention.

Let $\overline G \subset \mathrm{PGL}_2(\Qbar)$ denote the projective image of $\varrho_f$. It is well known that one of the following three possibilities occurs:
\begin{enumerate}
\item[(A)] $\overline G$ is the Klein four-group $\Z/2\Z \times \Z/2\Z$.
\item[(D)] $\overline G$ is the (non-abelian) dihedral group $D_r$ with $2r$ elements, where $r\geq 3$. 
%In this case, $\varrho_f=\Ind_K^F\chi$, where $K/F$ is a degree two extension and $\chi$ is a finite order character of~$G_K$. Note that in this case, $K$ is the unique such quadratic extension of $F$.
\item[(E)] $\overline G$ is exotic, that is, $\overline G$ is isomorphic to $A_4$, $S_4$, or $A_5$. Here, $A_r$ (resp. $S_r$) denote the alternating (resp. symmetric) group on $r$ letters. 
\end{enumerate} 
If $\overline G$ is a dihedral group, then $\varrho_f=\Ind_K^F\chi$, where $K/F$ is a degree two extension and $\chi$ is a finite order character of~$G_K$. Moreover if $\overline G$ is a non-abelian dihedral group, then such a $K$ is unique.
We say that $x$ is dihedral (resp. exotic) if the corresponding $\overline G$ is non-abelian dihedral (resp. exotic). We say that $x$ is a CM point if it is dihedral and $K/F$ is a totally complex extension.

When $F=\Q$, it was known from the work of Cho and Vatsal \cite{CV03} that  the weight map is {ramified} at $x$ if $f$ is $p$-regular  and has RM by a quadratic real field in which $p$ splits. Recently, Bella\"iche and Dimitrov \cite{D-B} showed that when $F=\Q$ and $f$ is a $p$-regular cusp form of weight one, the Coleman-Mazur eigencurve is smooth at $x$, and gave a precise criterion for the \'etaleness of the weight map.

For a general totally real field $F$, Betina \cite{Bet} and Deo \cite{Deo} have obtained results regarding the smoothness of $\cE$ and \'etaleness of $w$ at $x$ which is $p$-regular and dihedral. These results depend either on the splitting behavior of the primes of $F$ above $p$ in the extension $K/F$ or on the number of real embeddings of $K$. Most of these results depend either on Leopoldt's conjecture or on the $p$-adic Schanuel conjecture. To the best of our knowledge, no other work has been done on this topic in the case of $F \ne \Q$. In particular, no results about the geometry of $\cE$ at an exotic point are known so far.

In this paper, we study the smoothness of $\cE$ and the \'etaleness of the weight map $w$ at an exotic or CM classical $p$-regular point of parallel weight $1$. 
More precisely, we give a criterion for the \'etaleness of the weight map at an exotic or CM classical $p$-regular point of parallel weight $1$ when $F$ is Galois over $\Q$ in terms of the non-vanishing of the determinant of a certain $p$-adic regulator matrix $\mathrm{R}_p(f) \in \mathrm{M}_d(\overline{\Q}_p)$ introduced in Definition \ref{p-adicregulator}. On the other hand, without assuming the hypothesis that $F/\Q$ is Galois, we give a criterion for the smoothness of~$\cE$ in the CM case in terms of the non-vanishing of the determinant of another $p$-adic regulator matrix $\mathrm{R}_p(\chi^{-}) \in \mathrm{M}_d(\overline{\Q}_p)$ introduced in Definition \ref{def: cmreg} (see \S\ref{section: smoothness} for the definition of $\chi^{-}$). Moreover, we prove that the $p$-adic regulator matrices $\mathrm{R}_p(f)$ and $\mathrm{R}_p(\chi^{-})$ are invertible after assuming the $p$-adic Schanuel conjecture. In fact, the coefficients of both $\mathrm{R}_p(f)$ and $\mathrm{R}_p(\chi^{-})$ are $p$-adic logarithms of certain global algebraic units and their construction makes them amenable for the use of $p$-adic Schanuel conjecture. %While the definitions of $\mathrm{R}_p(f)$ and  $\mathrm{R}_p(\chi^{-})$ depend on the choice of these units, we always expect their ranks to be independent of that choice (see Proposition~\ref{non-vanishingregulator} and Lemma~\ref{lemma: chichisigmaiso}). 
%Moreover, we show in Proposition~\ref{non-vanishingregulator} and Lemma~\ref{lemma: chichisigmaiso} that the $p$-adic Schanuel Conjecture (see \cite[Conj.1]{Deo}) {predicts} always that $\det(\mathrm{R}_p(f)) \ne 0$ and $\det(\mathrm{R}_p(\chi^{-})) \ne 0$.

We will now state the main theorems of this article and their corollaries. 
\begin{theoremletter}\label{theorem: main1}
Let $f$ be a $p$-regular classical cuspidal Hilbert eigenform over $F$ of tame level $\gn$ and parallel weight~$1$, and let $x$ denote the point of $\cE$ corresponding to {an ordinary} $p$-stabilization $\tilde{f}$ of $f$. Suppose that $F/\Q$ is Galois, that Leopoldt's conjecture holds for $F$ and that $x$ is either exotic or CM. Let $\mathrm{R}_p(f) \in \mathrm{M}_d(\overline{\Q}_p)$ be the $p$-adic regulator matrix introduced in Definition~\ref{p-adicregulator}. Then:
\begin{enumerate}
\item The weight map $w : \cE \rightarrow \cW$ is \'etale at $x$ if $\det(\mathrm{R}_p(f)) \ne 0$.
\item If the $p$-adic Schanuel conjecture is true, then $\det(\mathrm{R}_p(f)) \ne 0$.
\end{enumerate}
\end{theoremletter}

 In Corollary~\ref{corfour}, we show that when $[F:\Q] =2 $, the determinant of $\mathrm{R}_p(f) $ is non-zero if the strong Four exponential conjecture (which is weaker than the $p$-adic Schanuel conjecture) is true.

Let $S_1(\gn p ,\chi_f)[\tilde{f}]$ (resp. $S^{\dag}_{1}(\gn, \chi_f)\lsem \tilde{f}\rsem$) denote the eigenspace (resp. generalized eigenspace) attached to $\tilde{f}$ inside the space of classical (resp. overconvergent) cuspidal Hilbert modular forms of parallel weight~$1$, level~$\gn p$ (resp. $\gn$), and nebentypus~$\chi_f$. 
Using the \'etaleness result of Theorem~\ref{theorem: main1} along with the arguments of \cite[\S.7]{Bet}, we get:
%Under the hypotheses of Theorem \ref{theorem: main1} and that $f$ is a newform of level $\gn$, one has $\dim S^{\dag}_{1}(\gn, \chi_f)\llbracket \tilde{f}\rrbracket=1$ (it follows from the etaleness of $w$ at $x$, see \cite[\S.7]{Bet}). Thus, we get:
%Thus, the natural inclusion \[S_{1}(\gn p, \chi_f)[\tilde{f}]  \overset{\simeq}{\hookrightarrow} S^{\dag}_{1}(\gn, \chi_f)\llbracket \tilde{f} \rrbracket \] is in fact an isomorphism.

\begin{corollaryletter}\label{theorem: main3} Assume the same hypotheses as in Theorem~\ref{theorem: main1}. If $f$ is a newform of level $\gn$ and $\det(\mathrm{R}_p(f)) \ne 0$ (which is implied by $p$-adic Schanuel conjecture), then 
$$
S_{1}(\gn p, \chi_f)[\tilde{f}]  \simeq  S^{\dag}_{1}(\gn, \chi_f)\lsem \tilde{f} \rsem \,.
$$ 
\end{corollaryletter}  

%Let $H$ denote the Galois closure of $\Qbar^{\ker(\ad^0 \varrho_f)}$, the splitting field of $\ad^0 \varrho_f$, over $\Q$.
Let $H$ denote the Galois closure over $\Q$ of the subfield of $\Qbar$ fixed by $\ker(\ad \varrho_f)$, where $\ad\varrho_f$ is the adjoint representation of $\varrho_f$.
\begin{theoremletter}\label{theorem: main2}
Let $f$ be a $p$-regular classical cuspidal CM Hilbert modular eigenform over $F$ of tame level $\gn$ and parallel weight~$1$, and $\mathrm{R}_p(\chi^{-}) \in \mathrm{M}_d(\overline{\Q}_p)$ be the $p$-adic regulator matrix introduced in Definition~\ref{def: cmreg}. Suppose Leopoldt's conjecture is true for $F$. Let $x$ be the point of $\cE$ corresponding to an ordinary $p$-stabilization of $f$. Then:
\begin{enumerate}
\item Let $K$ be the unique quadratic extension of $F$ such that $\varrho_f \simeq \Ind_{K}^{F}\chi$. If all the primes of $F$ lying above $p$ are either inert or ramified in $K$ {and Leopoldt's conjecture is true for $H$}, then $\cE$ is smooth at $x$.
\item If $\det(\mathrm{R}_p(\chi^{-})) \neq 0$, then $\cE$ is smooth at $x$. 
\item If $p$-adic Schanuel conjecture is true, then $\det(\mathrm{R}_p(\chi^{-})) \neq 0$.
\end{enumerate}
\end{theoremletter}

We now describe an application of our results to Hida theory.
When $f$ is a newform of tame level $\gn$, it has been shown in \cite[Prop.6.8]{Bet} that {if $\cE$ is smooth at $x$, then the full $p$-adic eigenvariety $\cE^{\full}$ is smooth at the point $x^\full$ corresponding to {the ordinary} $p$-stabilization of $f$ giving $x$.}
The $p$-adic nearly ordinary Hida Hecke algebra of level $\gn p^{\infty}$ is an integral formal model of the open-closed subset of $\cE^{\full} $ defined by $\mid U_p \mid_p=1$.
Combining these two facts with Theorem~\ref{theorem: main1} and Theorem~\ref{theorem: main2}, we get:
% Thus, there exists under the hypotheses of either Theorem~\ref{theorem: main1} or Theorem~\ref{theorem: main2} (and that $f$ is a newform) a unique nearly ordinary Hida family, up to Galois conjugacy, passing through a given {a ordinary $p$-stabilization of~$f$}.

\begin{corollaryletter}\label{theorem: main4} Suppose $f$ is a cuspidal Hilbert modular newform of tame level $\gn$.
Assume that one of the following conditions hold:
\begin{enumerate}
 \item Hypotheses of Theorem~\ref{theorem: main1} are satisfied and $\det(\mathrm{R}_p(f)) \neq 0$,
\item Hypotheses of Theorem~\ref{theorem: main2} are satisfied and either no prime of $F$ lying above $p$ is split in $K$ {and Leopoldt's conjecture is true for $H$} or $\det(\mathrm{R}_p(\chi^{-})) \neq 0$.
\end{enumerate}
Then there exists a unique nearly ordinary Hida family, up to Galois conjugacy, passing through a given ordinary $p$-stabilization of~$f$.
\end{corollaryletter}

%Note that the previous results of \cite{Bet} and \cite{Deo} in the CM case were under some {constraints} on either the splitting behavior of the primes of $F$ above $p$ or the degree $d$. 
%{After assuming $p$-adic Schanuel conjecture, we show that, in the CM case, smoothness result should hold regardless of any condition on $p$ or $d$ and the weight map should be \'etale under the relatively mild hypothesis that $F$ be Galois over $\Q$.}
% remove all those hypotheses to prove the smoothness at CM points and replace them by the relatively mild hypothesis that $F$ be Galois over $\Q$ to prove the \'etaleness of the weight map at CM points. 

%{\bf {Eigenvarieties and Iwasawa theory:}} 
A main theme in the theory of $p$-adic families of eigenforms is to relate the geometry of the eigenvarieties to $p$-adic $L$-functions. 
%Restricting ourselves to the weight $1$ case, it is well known that establishing the smoothness of the eigenvariety is useful in defining the $p$-adic $L$-functions attached to weight $1$ modular forms (see \cite[Cor. 1.3]{D-B}). 
%On the other hand, in the non-smooth cases, some recent works (\cite{BDPo,Bet-Dim}) have studied the exceptional zeros of $p$-adic $L$-functions for certain Artin motives following a new geometric approach based on the complete determination of the geometry of the eigencurve at such points.
%a classical weight one crossing point.  
To conclude the paper, we give an application of our main results to  $p$-adic $L$-functions similar to the ones obtained in \cite{BDPo,Bet-Dim}.

Assume that $\varrho_f=\Ind_K^F\chi$, {where} $K$ is a CM field with Leopoldt defect $\delta_K$ and that every prime $\gp$ of $F$ above $p$ splits in $K$ into $v_\gp \cdot \overline{v_\gp}$ (this condition is often known to be the ``ordinary hypothesis for $K$ and $p$''). Denote by $\gc$ the conductor of $\chi$.
%$\chi^{\tau}/\chi$, where $\tau$ is a complex conjugation and $\chi^{\tau}$ is the character defined by $\chi^{\tau}(g) :=\chi(\tau g \tau^{-1})$ for $g \in G_K$. 

Katz (\cite{Katz}) and Hida-Tilouine (\cite{H-T}) constructed a measure $\mu_{\gc}$ on the 
 ray class group $\cC\ell_K(p^{\infty}\gc) $ which interpolates the $L$-values at $s=0$ of the Hecke $L$-functions $L(s,\widehat{\phi})$ of admissible Hecke characters $\phi:\cC\ell_K(p^{\infty}\gc) \to \overline{\Q}_p^{\times}$, where $\phi$ is the $p$-adic avatar of the Hecke character $\hat{\phi}:\cC\ell_K(p^{\infty}\gc) \to \C^{\times}$.  %Admissible means that the infinite type of $\hat{\phi}$  satisfies the condition ii) of \cite[Thm.II]{H-T}.
%The interpolation formula of \cite[Thm.II]{H-T} uniquely characterizes $\mu_{\gc}$, as these admissible characters are Zariski dense in $\Spec \cO \llbracket \cC\ell_K(p^{\infty}\gc) \rrbracket$, where $\cO$ is the ring of integers of {a suitably large finite extension of $\Q_p$.} 
However, the  finite order characters of $\cC\ell_K(p^{\infty}\gc)$ are not admissible (since their infinite type is trivial), and hence they are away from the range of classical interpolation area for the Katz measure $\mu_{\gc}$. 

When $F=\Q$, the $p$-adic analogue of Kronecker's second limit formula \cite[\S.1.5]{H-T94} yields that 
$\mu_{\gc}(\chi^{-}) \ne 0$ if and only if $\chi^{-}\mid_{ G_{\Q_p}} \ne 1$, where $\chi^{-}$ is the anti-cyclotomic character defined in \S\ref{section: smoothness}. However, there is no $p$-adic analogue of Kronecker's second limit formula for a general totally real field $F$. So it is not known whether $\mu_{\gc}(\chi^{-}) \ne 0$. 

Let $\tau$ be a complex conjugation, $\mathrm{tor}$ denote the torsion subgroup of $\cC\ell_K(p^{\infty}\gc)$ and $\zeta_{\chi^{-}}^{-} \in \Lambda_0:=\cO \lsem \cC\ell_K(p^{\infty}\gc)/\mathrm{tor}\rsem \simeq \cO \lsem X_1,..,X_{d+1+\delta_K} \rsem$ be the formal power series corresponding to the 
the push-forward of $\mu_\gc$ by 
\[\cC\ell_K(p^{\infty}\gc) \xrightarrow{z\to z^{\tau}/z}  
\cC\ell_K(p^{\infty}\gc) \overset{\pi_\chi}{\longrightarrow} \Lambda_0^\times,\]
where $\pi_\chi$ is the $\chi$-projection sending $z  \in \cC\ell_K(p^{\infty}\gc) \to \chi(z)[z] \in  \Lambda_0^\times $ and $z^{\tau} := \tau z \tau^{-1}$. The formal power series $\zeta_{\chi^{-}}^{-} \in \Lambda_0$ is often called the anti-cyclotomic Katz $p$-adic $L$-function.

\begin{corollaryletter}\label{theorem: main5}Suppose that Leopoldt's conjecture holds for $K$, that every prime $\gp$ of $F$ above $p$ splits in $K$ into $v_\gp \cdot \overline{v_\gp}$, that $\chi^{-}\mid_{ G_{K_{v_\gp}}} \ne 1$ for all primes $\gp \mid p$ of $F$ and that $\det(\mathrm{R}_p(\chi^{-})) \ne 0$. Then $\zeta_{\chi^{-}}^{-}(0) \ne 0$. In particular, one has $\mu_{\gc}(\chi^{-}) \ne 0$.

\end{corollaryletter}
We will prove this corollary in \S\ref{padiclfn}.

We now give an outline of the proofs of Theorems~\ref{theorem: main1} and \ref{theorem: main2}. Let $\mathcal{T}$ and $\Lambda$ denote the completed local rings of $\cE$ and $\cW$ at $x$ and $w(x)$, respectively. Let $\gm$ be the maximal ideal of $\Lambda$ and let $\cT'_0 = \cT/\gm\cT$ be the algebra of the fiber of $w$ at $x$. Note that $\cT'_0$ is an artinian ring, since $w$ is a locally finite morphism, while the ring $\cT$ has Krull dimension $d+1$ as $\cE$ is equidimensional of dimension $d+1$. Let $t_\cT$ (resp. $t_{\cT'_0}$) denote the tangent space of $\cT$ (resp. $\cT_0'$). Thus, proving that:
\begin{enumerate}
\item $\cE$ is smooth at $x$ is equivalent to showing that the tangent space $t_\cT$ has dimension $d+1$.
\item $w$ is \'etale at $x$ is equivalent to showing that the tangent space $t_{\cT'_0}$ is zero.
\end{enumerate}

To compute these dimensions, one relates the rings $\cT$ and $\cT'_0$ to the universal rings $\cR$ and $\cR'_0$ representing the nearly ordinary deformation functor $\cD$ of $\varrho_f$ and the functor of ordinary deformations with constant determinant $\cD'_0$ of $\varrho_f$, respectively. 
Let $\cD'$ be the constant determinant counterpart of $\cD$.
These functors, which will be introduced in \S\ref{section: def functors}, are representable precisely because of the $p$-regularity hypothesis and the fact that $\varrho_f$ is absolutely irreducible. In \S\ref{section: etaleness} and \S\ref{section: smoothness}, we show that the dimensions of the tangent spaces $t_{\cD'_0}$ and $t_{\cD}$ are $0$ and $d+1$ under the hypotheses of Theorem~\ref{theorem: main1} and \ref{theorem: main2}, respectively. 

For this computation of dimensions, we first follow the strategy of \cite{D-B} to get a description of the tangent spaces $t_{\cD'_0}$ and $t_{\cD'}$ in terms of elements of $(\Hom(G_H,\overline{\Q}_p) \otimes \mathrm{ad}^0\varrho_f)^{\Gal(H/F)}$ satisfying the suitable $p$-ordinary and $p$-nearly ordinary deformation conditions, respectively.%\[\rH^1(F,\ad^0\varrho_f) \simeq (\Hom(G_H,\bar{\Q}_p) \otimes \mathrm{ad}^0\varrho_f)^{\Gal(H/F)} \simeq (\Hom((\cO_H \otimes  \Z_p)^{\times}/\cO_H^{\times},\bar{\Q}_p) \otimes \mathrm{ad}^0\varrho_f)^{\Gal(H/F)}\] 
%where $H$ is the Galois closure of the number field cut out by the adjoint representation of $\varrho_f$.  

We first treat the case of Theorem~\ref{theorem: main1}. In this case, we study a certain non-torsion $\Z$-submodule of maximal rank of the group of global units of the field $H$ and use the techniques of \cite{D-B}, along with the description of elements of $t_{\cD'_0}$, to find a matrix $\mathrm{R}_p(f) \in \mathrm{M}_d(\overline{\Q}_p)$ such that an element of $t_{\cD'_0}$ corresponds to an element in the kernel of $(\mathrm{R}_p(f))^t$. This yields that $\dim t_{\cD'_0}=0$ when $\det(\mathrm{R}_p(f)) \ne 0$. 
In the case of Theorem~\ref{theorem: main2}, we again use the description of an element of $t_{\cD'}$ to find an explicit basis of $t_{\cD'}$. This process needs Leopoldt's conjecture in the first case of Theorem~\ref{theorem: main2}. While in the second case, we need the matrix $\mathrm{R}_p(\chi^{-})$ to be invertible for this process. A direct consequence of both computations is that $\dim t_{\cD}=d+1$. We also prove that the matrices $\mathrm{R}_p(f)$ and $\mathrm{R}_p(\chi^{-})$ are invertible if the $p$-adic Schanuel conjecture is true in \S\ref{section: etaleness} and \S\ref{section: smoothness}, respectively. For proving these results, we use the description of these matrices along with the results of \S\ref{section: def functors}. 

The proofs of the main theorems are based on proving an $\cR=\cT$ theorem, and are then completed in \S\ref{section: proof of AB}.
\subsection*{Notation}

To ease notations, with no further description, by~$\gp$ we will always mean a prime ideal of~$\cO_F$ dividing $p$. Throughout the article, $f$ will denote a $p$-regular cuspidal Hilbert modular eigenform over $F$ of parallel weight $1$, tame level~$\mathfrak{n}$ and nebentypus $\chi_f$. By $x$, we will refer to the point of $\cE$ corresponding to {an ordinary} $p$-stabilization of $f$. We will denote the absolute Galois group of a field $M$ by $G_M$.

\subsection{Acknowledgements} 
Betina was supported by the EPSRC Grant EP/R006563/1. This work was done when Deo was a postdoc at the Mathematics Research Unit of University of Luxembourg. Fit\'e was funded by the Excellence Program Mar\'ia de Maeztu MDM-2014-0445. Fit\'e was partially supported by MTM2015-63829-P. Fit\'e was partially supported by the Simons Foundation grant 550033. This project has received funding from the European Research Council (ERC) under the European Union's Horizon 2020 research and innovation programme (grant agreement No 682152).
We thank the anonymous referee for many useful suggestions which helped us in improving the exposition.

\section{Preliminaries}\label{section: def functors}

{In this section, we will define the relevant deformation functors and collect various results which will be used in proving the main theorems. %To be precise, in \S\ref{section: defdeffunctors}, we will define the nearly ordinary deformation functor $\cD$, its constant determinant counterpart $\cD'$ and the functor $\cD'_0$ interpolating ordinary deformations with constant determinant. In \S\ref{section: taspaces}, we will give cohomological characterizations of their tangent spaces and will also give their explicit descriptions. In \S\ref{sec: units}, we will collect some results about algebraic units in a suitable Galois extension of $\mathbb{Q}$. 
Since most of this material is well known, we will be rather brief in most parts of this section.}

\subsection{Deformation functors}\label{section: defdeffunctors}
Before proceeding further, let us describe the choice of a basis $(v_1,v_2)$ of $\overline \Q_p^2$, which we will use to regard the projective image $\overline G$ of $\varrho_f$ inside $\PGL_2(\Qbar)$. 

If $\overline G$ is exotic, then we choose a basis $(v_1,v_2)$ such that $\varrho_f(G_F) \subset \GL_2(\Qbar)$. If $\overline G$ is a non-abelian dihedral group, recall first that there exists a unique quadratic extension $K/F$ such that $\varrho_f \simeq \Ind_K^F \chi$ for some character $\chi$ of $G_K$. Fix a lift $\sigma$ of the non-trivial element of $\Gal(K/F)$ in $G_F$ and let $\chi^\sigma$ be the character of $G_K$ defined by $\chi^{\sigma}(g) :=\chi(\sigma g \sigma^{-1})$ for all $g \in G_K$. We then choose $(v_1,v_2)$ so that in this basis we have:
\begin{enumerate} 
\item $\varrho_f|_{G_K} = \chi \oplus \chi^{\sigma}$.
\item The projective image of $\varrho_f(\sigma)$ is $\begin{pmatrix}0 & 1\\ 1 & 0\end{pmatrix} \in \PGL_2(\Qbar)$. 
\end{enumerate}
Note that these two conditions imply that under the basis $(v_1,v_2)$, we have $\varrho_f(G_F) \subset \GL_2(\Qbar)$. We fix this choice of basis throughout the article unless mentioned otherwise. 

For each $\gp|p$, choose an embedding $\iota(\gp): \overline \Q \hookrightarrow \overline \Q_p$ such that the diagram
\begin{equation}\label{equation: chooseofembed}
\xymatrix{
F_\gp \ar@{^{(}->}[r] &  \Qbar_p\\
F \ar@{^{(}->}[u] \ar@{^{(}->}[r] &\Qbar\ar@{^{(}->}[u]_{\iota(\gp)}}
\end{equation}  
is commutative. Note that these embeddings might be different than the embedding $\iota_p$ fixed in the introduction. There exists a unique prime $\gp_0$ of $F$ dividing $p$ for which we can choose $\iota(\gp_0)$ to be $\iota_p$ and we make this choice.
Let $D_\gp$ (resp. $I_\gp$) denote the decomposition (resp. inertia) group at $\gp$. The choice of the embedding $\iota(\gp)$ provides an identification of $G_{F_\gp}$ with $D_\gp$ which is a subgroup of $G_F$. We will now use this identification for the rest of the article for restricting representations of $G_F$ to $G_{F_{\gp}}$.

Let $\tilde f$ be {an ordinary} $p$-stabilization of $f$ and let $x$ be the point on $\cE$ corresponding to it. Recall that, for any $\gp|p$ of $\cO_F$, there exists a basis $(v_{1,\gp},v_{2,\gp})$ of $\Qbar_p^2$ such that $v_{i,\gp} \in \Qbar v_1 \oplus \Qbar v_2$ for $i=1,2$ and under this basis
\begin{equation}\label{equation: in basis}
\varrho_f|_{G_{F_\gp}}= 
\Qbar_p(\psi'_\gp) \oplus \Qbar_p(\psi''_\gp) \,,
\end{equation}
where $\psi'_\gp: G_{F_\gp} \rightarrow \overline \Q_p^\times$ is a character and $\psi''_\gp: G_{F_\gp} \rightarrow \overline \Q_p^\times$ is an unramified character such that $\psi'_{\gp} \neq \psi''_{\gp}$ and $\psi''_\gp(\Frob_\gp)$ is the $U_\gp$-eigenvalue of $\tilde f$. Note that the $p$-regularity of $f$ implies that $\psi'_{\gp} \neq \psi''_{\gp}$.

\begin{defn}\label{defn: func}  Let $\gC$ denote the category of local artinian rings $R$ with maximal ideal $\gm_R$ and residue field $R/\gm_R\simeq \overline \Q_p$. Define functors $\cD, \cD', \cD'_0:\gC \rightarrow \mathrm{Sets}$ in the following manner. For $R\in \gC$, let:
\begin{enumerate}
\item $\cD(R)$ be the set of strict equivalence classes of representations $\varrho_R: G_F\rightarrow \GL_2(R)$ such that $\varrho_R \equiv \varrho_f \pmod {\gm_R}$ and $\varrho_R$ is \emph{nearly ordinary} at $p$. Recall that we say that $\varrho_R$ is nearly ordinary at $p$ if for every $\gp|p$ we have
$$
\varrho_R|_{G_{F_\gp}}\simeq \left( \begin{matrix}
\psi'_{\gp,R} & * \\
0 &\psi''_{\gp,R} \end{matrix}\right)\,,
$$
where $\psi''_{\gp,R}: G_{F_\gp}\rightarrow R^\times$ is a character such that $\psi''_{\gp,R}\equiv \psi''_{\gp} \pmod {\gm_R}$.
\item $\cD'(R)$ be the subset of $\cD(R)$ of strict equivalence classes of representations $\varrho_R$ such that $\det(\varrho_R)=\det(\varrho_f)$.
\item $\cD_0'(R)$ be the subset of $\cD'(R)$ of strict equivalence classes of representations $\varrho_R$ for which $\psi''_{\gp,R}$ is unramified.
\end{enumerate}
\end{defn}

\subsection{Tangent spaces}\label{section: taspaces}

We will first recall cohomological descriptions of the respective tangent spaces $t_\cD$, $t_{\cD'}$, and $t_{\cD_0'}$ of the functors of Definition~\ref{defn: func}.

{Let $\ad \varrho_f$ be the adjoint representation of $\varrho_f$, 
%that is, if $V$ is the module affording the representation $\varrho_f$, then $\ad \varrho_f$ is the representation of $G_F$ on $\End(V)$ defined by $\ad\varrho_f(g)(X)=\varrho_f (g)X\varrho_f(g)^{-1}$. 
and $\ad^0\varrho_f$ be the subrepresentation of $\ad\varrho_f$ given by the subspace of $\End(V)$ of trace zero endomorphisms. }

\begin{rem} Recall that if~$\overline G$ is exotic, then $\ad^0\varrho_f$ is irreducible. If $\overline G$ is dihedral, then $\ad^0\varrho_f=\epsilon_K\oplus\Ind_K^F\big(\chi/\chi^\sigma\big)$, where $\epsilon_K$ is the non-trivial character corresponding to the quadratic extension $K/F$. Moreover, if $\overline G$ is dihedral and non-abelian, then we have that $\Ind_K^F\big(\chi/\chi^\sigma\big)$ is irreducible.
\end{rem}

%For every $\gp|p$, by \eqref{equation: in basis} we have an isomorphism
%$$
%(a_\gp,b_{\gp},c_{\gp},d_{\gp}):\ad\varrho_f|_{G_{F_\gp}}\simeq \Qbar_p \oplus \Qbar_p(\psi_\gp'/\psi_\gp'') \oplus\Qbar_p(\psi_\gp''/\psi_\gp') \oplus\Qbar_p\,.
%$$ 
% Let $$c_{\gp,*}: \rH^1(F,\ad\varrho_f) \rightarrow \rH^1(F_\gp, \Qbar_p(\psi ''_\gp/\psi'_\gp))\, \qquad d_{\gp,*}: \rH^1(F,\ad\varrho_f) \rightarrow \rH^1(I_\gp, \Qbar_p)$$
%be the map obtained by composing the map induced in cohomology by $c_\gp$ with the restriction maps from $G_F$ to $G_{F_\gp}$. 

Note that $t_\cD \subset \rH^1(F,\ad\varrho_f)$ and $t_{\cD_0'} \subset t_{\cD'} \subset \rH^1(F,\ad\varrho^0_f)$. We aim now at a description of these tangent spaces which will be more amenable for computation. For this purpose, we will first set up some more notation and make some observations.

Let $H$ be the normal closure over $\Q$ of the subfield of~$\Qbar$ fixed by the kernel of $\ad \varrho_f$.
Now let $\gP:=\gP (\gp)$ be the prime of $H$ rendering the following diagram  
\begin{equation}\label{equation: chooseofp}
\xymatrix{
F_\gp \ar@{^{(}->}[r] & H_\gP \ar@{^{(}->}[r]& \Qbar_p\\
F \ar@{^{(}->}[u]\ar@{^{(}->}[r] & H \ar@{^{(}->}[u] \ar@{^{(}->}[r] &\Qbar\ar@{^{(}->}[u]_{\iota(\gp)}}
\end{equation} 
commutative.
Using ~\eqref{equation: chooseofp}, regard $G_{H_\gP} \subseteq G_{F_\gp}$ as the decomposition subgroup $D_\gP$ of $G_H$ at $\gP$. Let $I_\gP\subseteq D_\gP$ be the inertia subgroup at $\gP$.

Let $h_1,\dots,h_d:H\rightarrow H$ be lifts of the $d$ different embeddings of $F$ into $H$ satisfying that the diagram
\begin{equation}\label{rem: embchoice}
\xymatrix{
 H_\gP \ar@{^{(}->}[rr]^{h_{\gP}}&& \Qbar_p\\
 H \ar@{^{(}->}[u] \ar[r]^{h_i}& H\ar@{^{(}->}[r] &\Qbar\ar@{^{(}->}[u]_{\iota_p}}
\end{equation} 
is commutative for some prime $\gP$ of $H$ which is one of the primes chosen in diagram~\eqref{equation: chooseofp} and some $h_{\gP}$. As we have chosen $\iota(\gp)$ to be $\iota_p$ for a suitable prime $\gp$ in diagram~\eqref{equation: chooseofembed}, it follows that one of the $h_i$'s can be chosen to be the identity. We make this choice and assume without loss of generality that $h_1$ is identity.
When we want to emphasize that one of the primes $\gP=\gP(\gp)$ chosen above $\gp$ is relative to the embedding $h_i$, we will sometimes write $\gP(h_i)$.

Denote by $\cO_H$ the ring of integers of $H$ and by $G$ the Galois group $\Gal(H/F)$.
Observe that the inflation-restriction exact sequence applied to the extension $H/F$ yields the following isomorphism: 
\begin{equation}\label{equation: altchar}
\rH^1(F,\ad\varrho_f)\simeq \left(\Hom(G_H, \Qbar_p)\otimes \ad \varrho_f\right)^G\,,
\end{equation} 
where we view $\Hom(G_H, \Qbar_p)\otimes \ad \varrho_f$ as the space of matrices
\begin{equation}\label{equation: space of mat}
\begin{pmatrix}
a & b \\
c & d \end{pmatrix}\qquad \text{with}\qquad a,b,c,d \in \Hom(G_H,\Qbar_p)\,,
\end{equation}
equipped with an action of $G$ given by
\begin{equation}\label{equation: invariant}
g\cdot  \begin{pmatrix}
a & b \\
c & d \end{pmatrix}= \varrho_f (g) \begin{pmatrix}
g\cdot a & g\cdot b \\
g\cdot c & g\cdot d \end{pmatrix}\varrho_f (g) ^{-1}\,.
\end{equation}

Note that, we have the exact sequence of $\Qbar_p[G]$-modules
\begin{equation}\label{equation: global-local}
0\longrightarrow \Hom(G_H,\Qbar_p) \longrightarrow  \Hom \left( (\cO_{H}\otimes \Z_p)^\times, \Qbar_p\right)\longrightarrow \Hom \left(\cO_{H}^{\times},\Qbar_p \right),
\end{equation}
coming from global class field theory. %The first map is dual to the Artin reciprocity map, and the second is the restriction with respect to the diagonal inclusion $\cO_{H}^{\times}\hookrightarrow \left(\cO_H \otimes \Z_p \right)^{\times}$.
%We will consider now the inclusion
%\begin{equation}\label{equation: inclusionseq}
%\Hom(G_H,\Qbar_p) \subseteq  \Hom \left( (\cO_{H}\otimes \Z_p)^\times, \Qbar_p\right)
%\end{equation}
%coming from the first part of \eqref{equation: global-local} (we will exploit the second part of \eqref{equation: global-local} in the next subsection). 

Let $\log_p:\Qbar_p^{\times}\rightarrow \Qbar_p$ be the standard $p$-adic logarithm sending $p$ to $0$. Note that $\bigsqcup_{i= 1}^d \iota_p \circ h_i G$ is a partition of the set of embeddings of $H$ into $\Qbar_p$, and that 
$$
(L_{i,g})_{i=1,\dots, d, g\in G}\,,\qquad \text{where }L_{i,g}:=\log_ p(\iota_p \circ h_i\circ g^{-1} (\cdot))\,,
$$
is a basis of $\Hom((\cO_H\otimes \Z_p)^\times,\Qbar_p)$. %In this basis, the $\Qbar_p[G]$-module structure of the space $\Hom \left( (\cO_{H}\otimes \Z_p)^\times, \Qbar_p\right)$ admits a particularly easy description ({see \cite[p. 11]{Deo} for more details}):

\begin{lemma}\label{lemma: cft}
We have an isomorphism of $\Qbar_p[G]$-modules:
$$\Hom((\cO_H\otimes \Z_p)^\times,\Qbar_p)\simeq \bigoplus_{i=1}^d \Qbar_p[G]\simeq \bigoplus_\pi \pi^{d \dim \pi}\,,$$
given by sending a homomorphism of the form $t\colon u \otimes 1 \mapsto \sum_{i=1}^d\sum_{g \in G} c_{i,g} L_{i,g}(u)$, for some $c_{i,g}\in \Qbar_p$, to the element $\sum_{i=1}^d\sum_{g\in G}c_ {i,g}g$.
\end{lemma}
See \cite[p. 11]{Deo} for more details.

{Combining \eqref{equation: altchar}, \eqref{equation: space of mat}, \eqref{equation: invariant}, \eqref{equation: global-local} and Lemma~\ref{lemma: cft}, it follows that if $t \in t_{\cD'}$, then we can write 
$$t =\sum_{i=1}^{d} \sum_{g \in G} \varrho_f(g)  \begin{pmatrix} a_{i} & b_{i}\\ c_{i} & -a_{i} \end{pmatrix} \varrho_f(g)^{-1} L_{i,g} \qquad \text{where} \qquad a_{i}, b_{i}, c_{i} \in \Qbar_p.$$}

Now, we give a more precise description of an element of $t_{\cD'_0}$ in the form given above (see \cite[Prop. 2]{Deo} and its proof for more details):

% By using the inflation-restriction exact sequence, we then get the following lemma (see \cite[Lem. 3]{Deo}).

%\begin{lemma}\label{lemma: adtan} We have: 
%\begin{enumerate}
%\item $t_{\cD} \simeq \ker \left(\left(\Hom(G_H,\Qbar_p)\otimes \ad \varrho_f\right)^G \xrightarrow{\prod_{\gp|p}c_{\gp,*}} \prod_{\gp|p} \rH^1(H_\gP, \Qbar_p) \right)$.
%\item $t_{\cD'} \simeq \ker \left(\left(\Hom(G_H,\Qbar_p)\otimes \ad^0 \varrho_f\right)^G \xrightarrow{\prod_{\gp|p}c_{\gp,*}} \prod_{\gp|p} \rH^1(H_\gP, \Qbar_p) \right)$.

%\item  $t_{\cD'_0} \simeq \ker \left(\left(\Hom(G_H,\Qbar_p)\otimes \ad^0 \varrho_f\right)^G \xrightarrow{\prod_{\gp|p}(c_{\gp,*},d_{\gp,*})} \prod_{\gp|p} \rH^1(H_\gP, \Qbar_p)\oplus \rH^1(I_\gP,\Qbar_p) \right)$. 
%\end{enumerate}
%\end{lemma}

%Using the discussion so far, we obtain the following consequence of Lemma~\ref{lemma: adtan}.

\begin{lemma}\label{corollary: tanelement}
There exist matrices $M_1, \dots, M_d\in \GL_2 (\Qbar)$ such that
%\item Any $t\in t_{\cD'}$ can be written in the form
%$$
%t=\sum_{i=1}^d\sum_{g\in G} 
%\varrho_f(g)M_i
%\begin{pmatrix}
%a_i & b_i \\
%0 & -a_i \end{pmatrix}
%M_i^{-1}\varrho_f(g)^{-1}
 %L_{i,g},
%$$
%where $a_1,\dots,a_d,b_{1},\dots, b_{d}\in \Qbar_p$.
 any $t\in t_{\cD'_0}$ can be written in the form
$$
t=\sum_{i=1}^d\sum_{g\in G} 
\varrho_f(g)M_i
 \begin{pmatrix}
0 & b_i \\
0 & 0 \end{pmatrix}
M_i^{-1}\varrho_f(g)^{-1}
 L_{i,g},$$
where $b_{1},\dots, b_{d}\in \Qbar_p$.
\end{lemma}

Suppose $\overline G$ is dihedral. We will now give a description of elements of $t_{\cD'}$ in this case. We first define a partition $\{1,\dots,d\}=I_1 \cup I_2 \cup I_3$ in the following manner. Recall the notation $\gP(h_i)$ appearing just below \eqref{rem: embchoice}.
We set $i\in I_1$ if the prime $\gp$ of $F$ lying below $\gP(h_i)$ is split in $K$ and such that $\chi^{\sigma}|_{G_{F_\gp}}=\psi''_{\gp}$ and we set $i\in I_2$ if the prime $\gp$ of $F$ lying below $\gP(h_i)$ is split in $K$ and such that $\chi|_{G_{F_\gp}}=\psi''_{\gp}$, where $\psi''_{\gp}$ is as in \eqref{equation: in basis}. Finally, we define $I_3$ to be the subset of indices $i$ for which the prime $\gp$ of $F$ lying below $\gP(h_i)$ is either inert or ramified in $K$.

\begin{lemma}\label{lemma: specshape}
Let $M_1,\dots, M_d\in \GL_2(\Qbar)$ be the matrices chosen in Lemma~\ref{corollary: tanelement}. Then:
\begin{enumerate}
\item For each $i\in I_3$, there exists a unique homogeneous linear polynomial $\mathcal{H}_i(X,Y)\in \Qbar_p[X,Y]$ such that for every $a,b\in \Qbar_p$, there exist $x,y\in \Qbar_p$ such that
$$
\begin{pmatrix}
\mathcal{H}_i(x,y) & x \\
y & -\mathcal{H}_i(x,y)
\end{pmatrix}=M_i\begin{pmatrix}
a & b \\
0 & -a
\end{pmatrix}M_i^{-1}\,.
$$ 
\item {If $t \in t_{\cD'}$, then
$$
t=\sum_{i\in I_1,\,g\in G} 
\varrho_f(g)A_i\varrho_f(g)^{-1}
 L_{i,g} + \sum_{i\in I_2,\,g\in G} 
\varrho_f(g)B_i\varrho_f(g)^{-1}
 L_{i,g}+\sum_{i\in I_3,\,g\in G} 
\varrho_f(g)C_i \varrho_f(g)^{-1}
 L_{i,g}\,,
$$
where 
\begin{equation}\label{equation: ABC}
A_i:=\begin{pmatrix}
a_i & b_i\\
0 & -a_i
\end{pmatrix},\quad 
B_i:=\begin{pmatrix}
a_i & 0\\
b_i & -a_i
\end{pmatrix}, \quad 
C_i:=\begin{pmatrix}
\mathcal{H}_i(b_i,a_i) & b_i\\
a_i & -\mathcal{H}_i(b_i,a_i)
\end{pmatrix} 
\end{equation}
and $a_i,b_i\in \Qbar_p$. }
\end{enumerate}
\end{lemma}

\begin{proof}
  Note that the existence and the uniqueness of the $\mathcal{H}_i$'s follow from the specific shape of the matrices $M_i$ given in \cite[Eq. (5.3)]{Deo}. The description of $t \in t_{\cD'}$ follows directly from \cite[Eq. (5.3)]{Deo}.
\end{proof}

\begin{rem}\label{remark: matrix choice}
As explained in \cite[p. 3898--3900]{Deo}, in case that $\overline G$ is dihedral, under the basis $(v_1,v_2)$ that we have fixed, we get an isomorphism
\begin{equation}\label{equation: isoinv}
\left(\Hom(G_H,\Qbar_p)\otimes \ad^0\varrho_f\right)^G \simeq  \left(\Hom(G_H,\Qbar_p)\otimes \epsilon_K\right)^G \oplus \left(\Hom(G_H,\Qbar_p)\otimes \Ind_K ^F\big(\chi/\chi^\sigma\big)\right)^G 
\end{equation}
which is given by 
$$
t=\begin{pmatrix}
a & b \\
c & -a
\end{pmatrix}
\mapsto
(t^\sharp,t^\flat)=\left( \begin{pmatrix}
a & 0 \\
0 & -a
\end{pmatrix},
\begin{pmatrix}
0 & b \\
c & 0
\end{pmatrix}
\right)\,.
$$

\end{rem}

%\subsection{The module $\Hom(G_H,\Qbar_p)$.}\label{section: Homtan}

%The goal of this subsection is to describe the $\Qbar_p[G]$-module structure of $\Hom(G_H,\Qbar_p)$. 
%We recall that Minkowski's proof of Dirichlet's unit theorem (see \cite[p.12]{Deo}) describes the rightmost term of {the exact sequence} \eqref{equation: global-local} {of $G$-modules}.

%\begin{equation}\label{lemma: Dir}
%\Hom(\cO_H^{\times},\Qbar_p)\simeq \bigoplus_{i=1}^d \left(\Ind_{\{1,\tau_i\}}^G1\right)\setminus 1 \simeq \bigoplus_{\pi\not =1} \pi^{ n_\pi }\oplus 1^{d-1}\,, \end{equation}
%where $n_\pi=\sum_{i=1}^d\dim\pi^{\{i,+\}}$. It yields (combined with \eqref{lemma: cft}) the following isomorphism of $\bar{\Q}_p[G]$-modules:

%\begin{equation}\label{proposition: modstr}
%\Hom(G_H, \Qbar_p) \simeq \bigoplus_{\pi} \pi ^{m_\pi}\,,
%\end{equation}
%where $\pi$ runs over the irreducible representations of $G$, and 
%\[m_1 \geq 1 \qquad \text{and} \qquad   m_\pi \geq \sum\limits^{d}_{i=1} \dim \pi^{\{i,-\}} \text{ if } \pi\neq 1\,.\]

\subsection{On the $\ad^0 \varrho_f$-isotypic component of the units of $H$}\label{sec: units}

Let $h_1,\dots, h_d \in \Gal(H/\Q)$ be as in \S\ref{section: taspaces}. Recall that, we have assumed, without loss of generality, that $h_1$ is the identity. Fix an embedding $\iota : \Qbar \hookrightarrow \C$. For this choice of an embedding of $\Qbar$ into $\C$, let $\tau_1,\dots, \tau_d\in G$ be the complex conjugations associated to $h_1,\dots,h_d$. Hence, for every $1 \leq i \leq d$, $h_i^{-1}\tau_1h_i=\tau_i$. For $i=1,\dots, d$ and a finite dimensional representation $\pi$ of $G$, let $\pi^{\{i,-\}}$ (resp. $\pi^{\{i,+\}}$) denote the subspace of $\pi$ on which $\tau_i$ acts by $-1$ (resp. $+1$).

As a consequence of Minkowski's proof of Dirichlet's unit theorem, we have:
\begin{lemma}\label{lemma: Dir}
We have an isomorphism of $\Qbar_p[G]$-modules:
$$\Hom(\cO_H^{\times},\Qbar_p)\simeq \bigoplus_{i=1}^d \left(\Ind_{\{1,\tau_i\}}^G1\right)\setminus 1 \simeq \bigoplus_{\pi\not =1} \pi^{ n_\pi }\oplus 1^{d-1}\,,$$
where $n_\pi=\sum_{i=1}^d\dim\pi^{\{i,+\}}$.
\end{lemma}
See \cite[p. 12]{Deo} for more details.

\begin{rem}\label{remark: dimensioninds}
Note that if $\overline G$ is exotic, then for every $1 \leq i \leq d$ we have 
$
\dim_{\Qbar} \ad^0\varrho_f^{\{i,+\}}= 1 $. 
%,\qquad \dim_{\Qbar} \ad^0\varrho_f^{\{i,-\}}= 2\,.

If we are in the CM case (that is, $\overline G$ is a non-abelian dihedral group and $K/F$ is totally complex), then none of the $\tau_i$ belongs to $G_K$. Thus we have
$$
\dim_{\Qbar} \Ind_K^F(\chi/\chi^{\sigma})^{\{i,+\}}= 1\, \qquad \epsilon_K^{\{i,+\}}=0\,.
$$
%,\qquad \dim_{\Qbar} \Ind_K^F(\chi/\chi^{\sigma})^{\{i,-\}}= 1\,,\qquad \dim_{\Qbar} \epsilon_K^{\{i,-\}}=1\,, \qquad \epsilon_K^{\{i,+\}}=0\,.
\end{rem}

Note that as $\varrho_f$ is a totally odd representation, $H/\Q$ is a totally complex extension. Therefore $\rk(\cO_H^\times)=[H:\Q]/2-1$. The following result can be easily deduced from \cite[Thm. 3.26]{Nar04} and its proof, and may be regarded as a refinement of Lemma~\ref{lemma: Dir}. We will use it to obtain a non-torsion $\Z$-submodule of $\cO_H^\times$ of maximal rank which will be crucial for our purposes. 
 
\begin{lemma}\label{lemma: units basis}
There exists $\eta\in \cO_H^\times$ such that $\tau_1(\eta)=\eta$ and such that, if $S'$ denotes a set of representatives of left cosets of $\{1,\tau_1\}$ in $\Gal(H/\Q)$, then
\begin{equation}\label{equation: normrel}
\prod_{\bar g\in S' }  \bar g (\eta)=1\,
\end{equation}
is the only non-zero algebraic relation involving the elements of $\{ \bar g (\eta)\}_{\overline g \in S'}$. 
\end{lemma}
 %Such a unit is sometimes called a \emph{weak Minkowski unit}.
For the rest of this subsection, assume $F$ is Galois over $\Q$. 
%Thus, $G \unlhd \Gal(H/\Q)$ which means that $\{h_1^{-1},\cdots,h_d^{-1}\}$ is a set of representatives of left cosets of $G$ in $\Gal(H/\Q)$ (a priori it was only a set of representatives of right cosets of $G$ in $\Gal(H/\Q)$). 
Let $S$ be a set of representatives of left cosets of $\{1,\tau_1\}$ in $G$. 
Since $G$ is a normal subgroup of $\Gal(H/\Q)$, it follows that $\{h_i^{-1}\bar g\}_{1 \leq i \leq d, \bar g \in S}$ is a set of representatives of left cosets of $\{1,\tau_1\}$ in $\Gal(H/\Q)$.

For the rest of this subsection, let $\eta \in \cO_H^\times$ be the unit found in Lemma~\ref{lemma: units basis}. Consider the $\Z$-module 
\begin{equation}\label{eq: mod}
V_i = \left\{ \prod_{\bar g\in S} (h_i^{-1}\bar g (\eta))^{a_{\bar g}} \, | \, \sum_{\bar g \in S} a_ {\bar g}=0 \right\} \,.
\end{equation}

\begin{lemma}\label{invarianceofVi} $V_i$ is stable under the action of $G$ and we have an isomorphism of $\Qbar_p[G]$-modules:
\begin{equation}\label{equation: decVi}
V_i\otimes \Qbar\simeq \left(\Ind_{\{1,\tau_i\}}^G 1\right) \setminus 1 \simeq \bigoplus_{\pi} \pi ^{\dim \pi^{\{i,+\}}}\setminus 1 \,.
\end{equation}
\end{lemma}

\begin{proof}
It is easy to verify that $V_i$ is stable under the action of $G$ using the facts that $G \unlhd \Gal(H/\Q)$ and $\tau_1 \in G$.
Note that $h_i^{-1}\bar g (\eta)=h_i^{-1} g h_i ( h_i^{-1}(\eta))$ and $\tau_i( h_i^{-1}(\eta))=h_i^{-1}(\eta)$ as $h_i^{-1}\tau_1h_i=\tau_i$. Since {$G \unlhd \Gal(H/\Q)$}, $h_i^{-1} g h_i \in G$. 
For $1 \leq i \leq d$, let $S_i := \{h_i^{-1} \bar g h_i \mid \bar g \in S\}$. 
So $S_i$ is a set of representatives of left cosets of $\{1,\tau_i\}$ in $G$.
Thus we have $V_i = \left\{ \prod_{ g\in S_i} (g h_i^{-1} (\eta))^{a_{ g}} \, | \, \sum_{ g \in S_i} a_ { g}=0 \right\}$.
The rest of the lemma now follows from Lemma~\ref{lemma: units basis}.
\end{proof}
 
It is also plain from the definition of $V_i$ that if $h_ih_j^{-1}= h_k^{-1}g$ for some $g \in G$, then 
\begin{equation}\label{equation: Vitrans}
h_i(V_j)=V_k\,.
\end{equation}
%Note that, given $h_i$ and $h_j$, there exists a unique $h_k$ satisfying the condition $h_ih_j^{-1}= h_k^{-1}g$ for some $g \in G$.

Choose any element $\bar g_0$ in $S$. Let $\hat S$ be $S\setminus \{\bar g_0 \}$, so that it has cardinality $|S|-1=[H:\Q]/(2d)-1$. For $\bar g\in \hat S$, write 
$$
\eta_{i,\bar g}:=\frac{h_i^{-1}\bar g(\eta)}{h_i^{-1}\bar g_0(\eta)}\,.
$$ 
We will later require the following result, whose proof is immediate from Lemma~\ref{lemma: units basis} {after taking $S'=\{h_i^{-1}\bar g\}_{1 \leq i \leq d, \bar g \in S}$.}
\begin{lemma}\label{lemma: lilogsub}
We have that:
\begin{enumerate}
\item $(\eta_{i,\bar g})_{\bar g\in \hat S}$ is a $\Z$-basis for $V_i$. 
\item Even more, $(\eta_{i,\bar g})_{i=1,\dots,d, \bar g\in \hat S}$ is a $\Z$-basis for $\bigoplus_{i=1}^d V_i$.   
\end{enumerate}
In particular, if $i\not= j$, then $V_i \cap V_j=\{0\}$. 
\end{lemma}

\section{Tangent spaces of $\cD'_0$ and $\cD'$}
\subsection{Definition of $\mathrm{R}_p(f)$}
{The goal of this subsection is to define the $p$-adic regulator matrix $\mathrm{R}_p(f)$ appearing in Theorem~\ref{theorem: main1} under the hypotheses of Theorem~\ref{theorem: main1}, which we assume throughout this subsection. We first establish some background results which will give us the units constituting the entries of $\mathrm{R}_p(f)$. We keep the notation from \S\ref{sec: units}.}

For $i=1,\dots,d$, let $W_{i,\Qbar}\simeq \Qbar[G]$ denote the $\Qbar[G]$-submodule of $\Hom \left( (\cO_{H}\otimes \Z_p)^\times, \Qbar_p\right)$ obtained as the $\Qbar$-linear span of $(L_{i,g})_{g\in G}$. Denote by $K_{i,\Qbar}$ the kernel of the map
$W_{i,\Qbar} \rightarrow \Hom(V_{i}, \Qbar_p)\,,$
which is obtained by restriction with respect to the diagonal inclusion $V_{i} \hookrightarrow (\cO_H\otimes \Z_p)^\times$, {where $V_i$ is as defined in \eqref{eq: mod}.}
We will now prove a result similar to \cite[Thm. 3.5]{D-B} for $K_{i,\Qbar}$.

%Similarly, let $W_{i,\Qbar}\simeq \Qbar[G]$ denote the $\Qbar[G]$-submodule of $W_i$ obtained as the $\Qbar$-linear span of $(L_{i,g})_{g\in G}$ and let $K_{i,\Qbar}$ denote the kernel of $W_{i,\Qbar}\rightarrow \Hom(V_{i}, \Qbar_p)$ obtained again by restriction. There is no reason to expect that $K_{i}\simeq K_{i,\Qbar}\otimes _{\Qbar}\Qbar_p$, and in fact we have the following result. 

\begin{lemma}\label{lemma: kerneldec}
For $i=1,\dots,d$, we have that 
$$
K_{i,\Qbar}\simeq \bigoplus_{\pi=1 \text{ or } \pi^{\{i,+ \}}=0}\pi ^{\dim\pi}\,.
$$  
\end{lemma}

\begin{proof}
The proof is an adaptation of that of \cite[Thm. 3.5]{D-B}. As explained in loc. cit., we have a biequivariant structural morphism
\begin{equation}\label{equation: biequiv}
W_{i,\Qbar} \simeq \Qbar[G]\rightarrow \Hom (V_{i}\otimes \Qbar, V_{i} \otimes \Qbar)\,.
\end{equation}
By the Baker-Brumer theorem (see \cite[Thm. 1]{Br67}), the map 
$$
V_{i} \otimes \Qbar \rightarrow \Qbar_p\,,\qquad u\otimes v\mapsto \log_p(\iota_p (h_i(u))) \iota_p(v)
$$
is injective. Therefore $K_{i,\Qbar}$ coincides with the kernel of \eqref{equation: biequiv} and hence, it is both a left and a right $\Qbar[G]$-module.
Hence, it follows, from standard facts of representation theory of finite groups, that if $V$ is an irreducible $\Qbar[G]$-submodule of $K_{i,\Qbar}$, then $V$ does not appear in $V_i \otimes \Qbar$ (see proof of \cite[Thm. 3.5]{D-B} for more details).
The lemma now follows from combining this with \eqref{equation: decVi}. %But it is a standard fact in the theory of representations of finite groups that, for any right $\Qbar[G]$-module~$V$, the left $\Qbar[G]$-module $\ker(\Qbar[G]\rightarrow \Hom_{\Qbar}(V,V))$ is isomorphic to $\bigoplus_\pi \pi^{\dim \pi}$, where the sum runs over the $\pi$ that do not appear in $V$. To conclude the proof of the lemma use \eqref{equation: decVi}.  
\end{proof}

Recall that by Lemma~\ref{corollary: tanelement}, any element $t\in t_{\cD'_0}$ can be written in the form
$$
t=\sum_{i= 1}^d {b_{i}}
 \begin{pmatrix}
\alpha_{i} & \beta_{i} \\
\gamma_{i} & -\alpha_i \end{pmatrix}, \text{ with}
\begin{pmatrix}
\alpha_{i} & \beta_{i} \\
\gamma_{i} & -\alpha_i \end{pmatrix}
:=
\sum_{g\in G} 
\varrho_f(g)M_i
\begin{pmatrix}
0 & 1 \\
0 & 0 \end{pmatrix}
M_i^{-1}\varrho_f(g)^{-1}
L_{i,g},
$$
where $b_{1},\dots, b_{d}\in \Qbar_p$ and $M_1, \dots, M_d\in \GL_2 (\Qbar)$.

\begin{prop}\label{proposition: nonvanunit} 
For $i=1,\dots, d$, there exists a unit $u_i\in V_{i}$ such that $\beta_i(u_i)\not =0$.
\end{prop}

\begin{proof}
%The argument to prove the proposition varies slightly depending on whether we are in the exotic or the CM case. We first claim that $\beta_i\in W_{i,\Qbar}$ is non-trivial. Suppose that we are in the exotic case. If $\beta_i$ was trivial, then the non-trivial $\Qbar[G]$-submodule of $\ad^0(\varrho_f)$ spanned by
%\begin{equation}\label{equation: elgen}
%M_i\begin{pmatrix}
%0 & 1 \\
%0 & 0 \end{pmatrix}M_i^{-1}
%\end{equation}
%would be of dimension at most $2$. But this {contradicts} the fact that $\ad^0(\varrho_f)$ is irreducible of dimension~$3$. Now suppose we are in the dihedral case. In this case, if $\beta_i$ is trivial, then the shape of the $\varrho_f(g)$, for $g\in G_F$, under the basis $(v_1,v_2)$ chosen implies that the element \eqref{equation: elgen} is diagonal (note that, in this basis, conjugation by $\varrho_f(\sigma)$ swaps the antidiagonal entries). But this is clearly not true which implies that $\beta_i \neq 0$ in this case as well.

Set $T:=\Qbar \alpha_i + \Qbar \beta_i + \Qbar \gamma_i$ (resp. $T:= \Qbar \beta_i + \Qbar \gamma_i$). We claim that
\begin{equation}\label{equation: betaispan}
T = \Qbar[G] \beta_i \simeq \ad^0 \varrho_f\qquad \left(\text{resp. } T = \Qbar[G] \beta_i \simeq \Ind_K^F(\chi/\chi^\sigma)\right)\,.
\end{equation}
 Indeed, recall that for any {finite dimensional} $\Qbar[G]$-modules $V$ and $W$, we have an isomorphism
$
\phi : (W \otimes V^* )^G\rightarrow \Hom_{\Qbar[G]} (V,W).
$
%given by sending an element $w\otimes f$, with $f\in V^*$ and $w \in W$, to the homomorphism $\phi(w\otimes f)$, defined by $\phi(w \otimes f)(v)=f(v) w$ for every $v\in V$.
We know that $\begin{pmatrix}
\alpha_{i} & \beta_{i} \\
\gamma_{i} & -\alpha_i \end{pmatrix} \in (W_{i,\Qbar} \otimes \ad^0 \varrho_f)^G$ and when $f$ is CM, 
$\begin{pmatrix}
0 & \beta_{i} \\
\gamma_{i} & 0 \end{pmatrix} \in (W_{i,\Qbar} \otimes \Ind_K^F(\chi/\chi^\sigma))^G$ (by Remark~\ref{remark: matrix choice}).
From this and using the selfduality of $\ad^0 \varrho_f$ (resp. $\Ind_K^F(\chi/\chi^\sigma)$), it becomes apparent that $T$ is a subrepresentation of $\Qbar[G]\simeq W_{i,\Qbar}$ and, in fact, that $T$ is isomorphic to a subrepresentation of $\ad^0 \varrho_f$ (resp. $\Ind_K^F(\chi/\chi^\sigma)$). 

Note that $T \neq \{0\}$. Indeed if $T=\{0\}$, then it implies that $M_i\begin{pmatrix}
0 & 1 \\
0 & 0 \end{pmatrix}M_i^{-1}$ is the zero matrix (resp. is diagonal) which is clearly not true. Now \eqref{equation: betaispan} follows from the fact that $\ad^0 \varrho_f$ (resp. $\Ind_K^F(\chi/\chi^\sigma)$) is irreducible.

Suppose now that for every $u \in V_{i}$, we have that $\beta_i(u)=0$, that is, $\beta_i\in K_{i,\Qbar}$. {It is immediate from \eqref{equation: betaispan} that $T=\Qbar[G]\beta_i \subset K_{i,\Qbar}$ and hence, $K_{i,\Qbar}$ has a $G$-sub-representation isomorphic to $\ad^0(\varrho_f)$ (resp. $\Ind_K^F(\chi/\chi^\sigma)$).}
%$\alpha_i$, $\gamma_i \in K_{i,\Qbar}$ (resp. $\gamma_i \in K_{i,\Qbar}$). Therefore we have a nontrivial element
%$$
%\begin{pmatrix}
%\alpha_i & \beta_i\\
%\gamma_i & -\alpha_i
%\end{pmatrix}\in\left( K_{i,\Qbar}\otimes \ad^0 \varrho_f \right) ^G\qquad  \left( \begin{pmatrix}
%0 & \beta_i\\
%\gamma_i & 0
%\end{pmatrix}\in\left( K_{i,\Qbar}\otimes \Ind_K^F(\chi/\chi^\sigma) \right) ^G\right)\,.
%$$
But Lemma~\ref{lemma: kerneldec}, Remark~\ref{remark: dimensioninds}, and Schur's lemma give a contradiction. 
%imply that 
%$$
%\left(K_{i,\Qbar}\otimes \ad^0 \varrho_f \right) ^G=0\qquad \left(\text{resp. }\left(K_{i,\Qbar}\otimes \Ind_K^F(\chi/\chi^\sigma) \right) ^G=0\right) \,,
%$$ 
%which is a contradiction.
\end{proof}

Note that there are $\beta_{i,g}\in \Qbar$ such that
$
\beta_i= \sum_{g\in G} \beta_{i,g}L_{i,g}\in W_{i,\Qbar}\,. 
$

\begin{defn}\label{p-adicregulator}
Let $u_j \in V_{j}$ be as in Proposition~\ref{proposition: nonvanunit} and define a $p$-adic regulator matrix 
$$
\mathrm{R}_p(f):=(x_{ij})_{1 \leq i,j \leq d}:=\big(\beta_i(u_j)\big)_{1 \leq i,j \leq d}=\left(\sum_{g\in G} \beta_{i,g} L_{i,g}(u_j)\right)_{1 \leq i,j \leq d}\,.
$$
\end{defn}

\subsection{Dimension of $t_{\cD'_0}$}\label{section: etaleness}
{We assume the hypotheses of Theorem~\ref{theorem: main1} throughout this subsection.}

We show in the following {lemma} that the dimension of $t_{\cD'_0}$ depends on the rank of  $\mathrm{R}_p(f)$.

\begin{lemma}
One always has $\dim t_{\cD'_0} \leq d-\mathrm{rk}(\mathrm{R}_p(f))$. In particular,  $t_{\cD'_0}=0$ when $\det(\mathrm{R}_p(f)) \ne 0$.
\end{lemma}

\begin{proof}

We recall that, by the exact sequence \eqref{equation: global-local}, we have that $t=\sum_{i= 1}^d {b_{i}}
 \begin{pmatrix}
\alpha_{i} & \beta_{i} \\
\gamma_{i} & -\alpha_i \end{pmatrix} \in t_{\cD'_0}$ implies that 
$$
B\cdot \mathrm{R}_p(f) =(0,\dots,0),
$$
where $B =(b_1,\dots, b_d) \in \Qbar_p^d$. The lemma follows directly from this.
%Thus, $B^t \in \ker(\mathrm{R}_p(f)^t)$, and hence the desired assertion follows form the fact that  $\dim \ker(\mathrm{R}_p(f)^t)=d-\mathrm{rk}(\mathrm{R}_p(f))$.
\end{proof}

We show in the next proposition that the non-vanishing of the determinant of $\mathrm{R}_p(f)$ is predicted by $p$-adic Schanuel conjecture. Thus, we would always expect that $t_{\cD'_0}=0$.

\begin{prop}\label{non-vanishingregulator}
Suppose that the $p$-adic Schanuel conjecture is true. Then, the determinant of $\mathrm{R}_p(f)$ does not vanish and in particular, $t_{\cD'_0}=0$.
\end{prop}

\begin{proof}
Note first that, as $h_1$ is the identity, by \eqref{equation: Vitrans} and part $i)$ of Lemma~\ref{lemma: lilogsub} the diagonal terms of~$\mathrm{R}_p(f)$ can be written in the form
\begin{equation}\label{equation: aii}
x_{ii}=\sum_{\bar g \in \hat S}\beta_{i,\bar g} \log_p(\iota_p(\eta_{1,\bar g}) )\,,
\end{equation}
for some $\beta_{i,\bar g}\in \Qbar$. Similarly, if $i\not = j$, we have that
\begin{equation}\label{equation: aij}
x_{ij}=\sum_{\bar g \in \hat S}\beta_{i,j,\bar g} \log_p(\iota_p(\eta_{l,\bar g}) )\,,
\end{equation}
where $\beta_{i,j,\bar g}\in \Qbar$ and $h_{i}h_{j}^{-1} \in h_l^{-1}G$ with $h_l \not = h_1$. Note that Proposition~\ref{proposition: nonvanunit} ensures that the product of diagonal terms $\prod_{i=1}^d x_{ii}$ does not vanish. Therefore, by equations \eqref{equation: aii} and \eqref{equation: aij}, we have that $\det(\mathrm{R}_p(f))$ is the sum of a \emph{nonvanishing} $\Qbar$-linear combination of terms of the form 
$$
\prod_{\bar g\in \hat S} \log_p(\iota_p(\eta_{1,\bar g}))^{n_{\bar g}}\,,\qquad\text{ for }n_{\bar g}\in \Z\text{ with }\sum_{\bar g\in \hat S}n_{\bar g}=d\,,
$$
and a $\Qbar$-linear combination of terms of the form
$$
\prod_{l=1}^{d}\prod_{\bar g\in \hat S} \log_p(\iota_p(\eta_{l,\bar g}))^{n_{l,\bar g}}\,,
$$
 for $n_{l,\bar g}\in \Z$ with $\sum_{l=1}^{d}\sum_{\bar g\in \hat S}n_{l,\bar g}=d$  and $n_{l,\bar g} \neq 0$  for some $l \neq 1$ and some $\bar g \in \hat S$. So, the first linear combination is a homogeneous $\Qbar$-polynomial in $\{\log_p(\eta_{1,\bar g})\}_{\bar g \in \hat S}$ of degree $d$ while the sum of their exponents occurring in each term of the second linear combination is strictly less than $d$.
Thus $\det(\mathrm{R}_p(f))$ is a \emph{nonzero} polynomial in $\{\log_p(\iota_p(\eta_{l,\bar g}))\}_{1 \leq l \leq d, \bar g \in \hat S}$ with coefficients in $\Qbar$. From part $ii)$ of Lemma~\ref{lemma: lilogsub}, we know that $\{\eta_{l,\bar g}\}_{1 \leq l \leq d, \bar g \in \hat S}$ is a set of algebraic units which are linearly independent over $\Z$. So if the $p$-adic Schanuel conjecture (\cite[Conj. 1]{Deo}) is true, then we get that $\{\log_p(\iota_p(\eta_{l,\bar g}))\}_{1 \leq l \leq d, \bar g \in \hat S}$ is a set of elements of $\Qbar_p$ which are algebraically independent over $\Q$. Therefore, we can now conclude that $\det(\mathrm{R}_p(f))$ is nonzero if the $p$-adic Schanuel conjecture is true (see \cite[\S 8, p. 3905]{Deo} for more details).
\end{proof}

As a consequence of the proof above, it follows that when $[F:\Q]=2$, we can conclude the non-vanishing of $\det(\mathrm{R}_p(f))$ assuming {a weaker conjecture.} We record this observation as a corollary:
\begin{cor}\label{corfour}
Suppose $[F:\Q]=2$ and the strong four exponential conjecture is true. Then, the determinant of $\mathrm{R}_p(f)$ does not vanish and in particular, $t_{\cD'_0}=0$.
\end{cor}
\begin{proof}

When $[F:\Q]=2$, we see, from \eqref{equation: aii}, \eqref{equation: aij} and Lemma~\ref{lemma: lilogsub}, that the two rows of $\mathrm{R}_p(f)$ are linearly independent over $\Q$, and the two columns of $\mathrm{R}_p(f)$ are also linearly independent over $\Q$. Hence, the strong four exponential conjecture (\cite[Conj. 1.7]{Bet-Dim}) predicts that $\det(\mathrm{R}_p(f)) \ne 0$.

\end{proof}

\subsection{Dimension of $t_{\cD'}$}\label{section: smoothness}

The goal of this section is to prove that $\dim_{\Qbar_p} t_{\cD'}\leq d$ under the hypotheses of Theorem~\ref{theorem: main2}, which we assume throughout. In particular, we have a decomposition $\ad^0\varrho_f=\epsilon_K\oplus \Ind_K^F\big(\chi/\chi^{\sigma} \big)$. Recall that, we have chosen a lift $\sigma$ of the non-trivial element of $\Gal(K/F)$ in $G_F$. By abuse of notation, we will also denote its image in $\Gal(H/\Q)$ by $\sigma$. Let~$C$ denote $\Gal(H/K)$, and note that $\sigma\in G\setminus C$. Let $\chi^{-}$ be the anti-cyclotomic character $\chi^{\sigma}/\chi$. So we have $\Ind_K^F\big(\chi/\chi^{\sigma} \big) = \Ind_K^F\chi^{-}$. We will use this notation throughout this subsection.
We keep the notations established in Lemma~\ref{lemma: specshape}.
 
We define the space 
\begin{equation}\label{equation: defW}
W \subseteq \big(\Hom \left( (\cO_{H}\otimes \Z_p)^\times, \Qbar_p\right)\otimes \ad^0\varrho_f \big)^G
\end{equation} 
made of elements of the form
$$
t=\sum_{i\in I_1,\,g\in G} 
\varrho_f(g)A_i\varrho_f(g)^{-1}
 L_{i,g} + \sum_{i\in I_2,\,g\in G} 
\varrho_f(g)B_i\varrho_f(g)^{-1}
 L_{i,g}+\sum_{i\in I_3,\,g\in G} 
\varrho_f(g)C_i\varrho_f(g)^{-1}
 L_{i,g}\,,
$$
where 
$A_i$, $B_i$ and $C_i$ are as in \eqref{equation: ABC}.
We will write
$V:=W \cap \big( \Hom(G_H,\Qbar_p)\otimes \ad^0\varrho_f \big)^G\,.$
By Lemma~\ref{lemma: specshape}, we have $t_{\cD'} \subseteq V$.

We will accomplish the goal of this section by showing that $\dim_{\Qbar_p}V=d$.
We begin by treating the simplest case in the following lemma.
\begin{lemma}
\label{nonsplitlem}
If $I_1=I_2=\emptyset$ and Leopoldt's conjecture is true for $H$, then $\dim_{\Qbar_p}V=d$.
\end{lemma}
\begin{proof}
As $I_1=I_2=\emptyset$, it follows that if $t \in W$, then 
$$
t= \sum_{1 \leq i \leq n} \sum_{g\in G} \varrho_f(g)C_i\varrho_f(g)^{-1} L_{i,g}\,,
$$ 
where $C_i$ is as in \eqref{equation: ABC}. 
%and the $\mathcal{H}_i(X,Y)$'s are the polynomials found in Lemma~\ref{lemma: specshape}.
By Remark~\ref{remark: dimensioninds}, taking the $\epsilon_K$-isotypical component of each term of the exact sequence \eqref{equation: global-local} tensored with $\ad^0 \varrho_f$, yields an isomorphism
\begin{equation}\label{equation: eKisocomp}
\big(\Hom(G_H,\Qbar_p)\otimes \epsilon_K\big)^G\simeq \big(\Hom((\cO_H\otimes \Z_p)^\times,\Qbar_p)\otimes \epsilon_K\big)^G\,.
\end{equation} 
Therefore, $t \in V$ if and only if $t^\flat\in \big(\Hom(G_H,\Qbar_p)\otimes \Ind_K^F\chi^{-}\big)^G$, where $t^\flat$ is the antidiagonal component of $t$ as defined in Remark~\ref{remark: matrix choice}.
If Leopoldt's conjecture is true for $H$, then the last arrow of \eqref{equation: global-local} is surjective and hence, by Remark~\ref{remark: dimensioninds}, $\dim_{\Qbar_p}\big(\big(\Hom(G_H,\Qbar_p)\otimes \Ind_K^F\chi^{-}\big)^G\big)=d$. Note that we are free to choose the anti-diagonal entries of $C_i$ and they determine the diagonal entries of $C_i$ uniquely (see Lemma~\ref{lemma: specshape}). Hence, it follows that if $I_1=I_2=\emptyset$ and Leopoldt's conjecture is true for $H$, then $\dim_{\Qbar_p}V=d$.
\end{proof}

Now for the rest of this section, assume that $I_1 \cup I_2 \neq \emptyset$. Let $V_1$ be the subspace of elements of $V$ for which the matrices $A_i$, $B_i$, and $C_i$ are diagonal. 

\begin{prop}
We have that $\dim_{\Qbar_p}V_1 =|I_1|+|I_2|$. Moreover, $(t_j)_{j\in I_1\cup I_2}$, where
$$t_j =
\sum_{g\in G}\varrho_f(g)\begin{pmatrix}
1 & 0\\
0 & -1 
\end{pmatrix}\varrho_f(g)^{-1}L_{j,g}
 \in \big(\Hom \left( (\cO_{H}\otimes \Z_p)^\times, \Qbar_p\right)\otimes \ad^0\varrho_f \big)^G
$$
constitute a basis for $V_1$.
\end{prop}

\begin{proof}
Note that $t_j \in W$. Recall that we have fixed a basis $(v_1,v_2)$ of $\Qbar_p^2$ such that $\varrho_f(c)$ is diagonal for every $c \in C$ {and $\varrho_f(\sigma)$ is anti-diagonal.} 
%This implies that $t_j=t_j^\sharp$, where $t_j^\sharp$ is the diagonal component of $t_j$ as defined in Remark~\ref{remark: matrix choice}. 
This implies that $t_j \in \big(\Hom((\cO_H\otimes \Z_p)^\times,\Qbar_p)\otimes \epsilon_K\big)^G$, and the isomorphism \eqref{equation: eKisocomp} then shows that $t_j \in \big(\Hom(G_H,\Qbar_p)\otimes \epsilon_K\big)^G \subseteq \big(\Hom(G_H,\Qbar_p)\otimes \ad^0\varrho_f\big)^G$. We conclude that $t_j \in V_1$.

Since $t_j$ are linearly independent, we get that $\dim_{\Qbar_p}V_1\geq |I_1|+|I_2|$. On the other hand, if a matrix $C_i$ is diagonal, then it must be the zero matrix. Hence, we have $\dim_{\Qbar_p}V_1\leq |I_1|+|I_2|$ which concludes the proof.
%Provided that the $t_j$ are linearly independent, it only remains to show that $\dim_{\Qbar_p}V_1\leq |I_1|+|I_2|$. Let $W_1$ be the subspace of $W$ made of elements for which the matrices $A_i$, $B_i$, and $C_i$ are diagonal. If a matrix $C_i$ is diagonal, then it must be the zero matrix. Hence, we have that $W_1$ has dimension $|I_1|+|I_2|$. Now the fact that $V_1$ is contained in $W_1$ concludes the proof.
\end{proof}

Our goal now is to define a subspace $V_2\subseteq V$ such that $V_1\oplus V_2=V$ and such that $\dim_{\Qbar_p}V_2=|I_3|$. To this aim, we will require the next lemma. Let us first introduce the following notation. For $i\in \{1,\dots,d\}$, define
$$
\delta(i):=\begin{cases}
1 & \text{if $i\in I_2$}\\
0 & \text{ otherwise.}
\end{cases}
$$
We use the notation $\Res^F_K\cO_H^\times\otimes \Qbar$ in the following lemma to indicate that we are viewing $\cO_H^\times\otimes \Qbar$ as a $C$-representation.

Let $\cO_H^\times[\chi^{-}]$ be the $\chi^{-}$-isotypical component of $\Res^F_K\cO_H^\times\otimes \Qbar$. By Frobenius reciprocity, Remark~\ref{remark: dimensioninds} and Lemma~\ref{lemma: Dir} we get that \[
\dim \cO_H^\times[\chi^{-}]=\dim_{\Qbar} \Hom_C\big(\Res^F_K \cO_H^\times \otimes \Qbar, \chi^{-}\big)=\dim_{\Qbar} \Hom_G\big(\cO_H^\times \otimes \Qbar, \Ind_K^F\chi^{-} \big)=d\,.\]

\begin{lemma}\label{lem: basis}
{There exist units $u_1,\cdots,u_d \in \cO_H^{\times}$ such that $\{\sum_{c \in C} c^{-1}(u_i) \otimes \chi^{-}(c)\}_{1 \leq i \leq d}$ forms a basis for $\cO_H^\times[\chi^{-}]$.}
\end{lemma}
\begin{proof}
{Note that $\{h_1^{-1},\cdots,h_d^{-1}\}$ is a set of representatives of right cosets of $G$ in $\Gal(H/\Q)$. 
We claim that $S_0:=\{ch_i^{-1}\}_{c \in C, 1 \leq i \leq d}$ is a set of representatives of left cosets of $\{1,\tau_1\}$ in $\Gal(H/\Q)$.
To prove the claim, suppose $c_1h_i^{-1}$ and $c_2h_j^{-1}$ are in the same left coset for some $c_1,c_2\in C$ and $1 \leq i, j \leq d$.
 Then $(c_2h_j^{-1})^{-1}c_1h_i^{-1} = h_j(c_2^{-1}c_1)h_i^{-1} \in \{1,\tau_1\}$.
Conjugating by $h_j^{-1}$, we get $(c_2^{-1}c_1)h_i^{-1}h_j \in \{1,h_j^{-1}\tau_1h_j\}=\{1,\tau_j\} \subset G$.
 Since $c_2^{-1}c_1 \in C \subset G$, we have $h_i^{-1} \in Gh_j^{-1}$ which means $i=j$.
So, $c_2^{-1}c_1 \in \{1,\tau_j\}$. 
Since $K$ is a CM field, $C=\Gal(H/K)$ is a set of representative of left cosets of $\{1,\tau_i\}$ in $G$ (for any $1 \leq i \leq d$). 
As $c_2^{-1}c_1 \in C$, it follows that $c_2^{-1}c_1=1$. Thus no two distinct elements of $S_0$ belong to the same left coset of $\{1,\tau_1\}$ in $\Gal(H/\Q)$. 
Matching the cardinalities of both $S_0$ and $\Gal(H/\Q)/\{1,\tau_1\}$ proves our claim.}

{Let $\eta \in \cO_{H}^{\times}$ be the unit found in Lemma~\ref{lemma: units basis} and $u_i := h_i^{-1}(\eta)$. 
Let $u_{i,\chi^{-}}:=\sum_{c \in C} c^{-1}(u_i) \otimes \chi^{-}(c)$. 
So $u_{i,\chi^{-}} \in \cO_H^{\times}[\chi^{-}]$ for all $1 \leq i \leq d$.
By Lemma~\ref{lemma: units basis}, it follows that the elements $u_{1,\chi^{-}},\cdots,u_{d,\chi^{-}}$ are linearly independent over $\Qbar$.
Since $\cO_H^\times[\chi^{-}]$ has dimension $d$, we conclude that $\{u_{i,\chi^{-}}, 1 \leq i \leq d\}$ forms a basis of $\cO_H^\times[\chi^{-}]$.}

%Suppose $F \neq \Q$. Then Lemma~\ref{lemma: units basis} implies that, for each $1\leq i \leq d$, $\tau_i(u_i)=u_i$ and the elements of the set $S_i := \{ c(u_i), 1\leq i\leq d, c \in C \}$ are linearly independent over $\Q$.
%Moreover, the $\Qbar$ vector subspace of $\cO_{H}^{\times} \otimes \Qbar$ generated by the set $S'_i := \{ c(u_i) \otimes 1, 1\leq i\leq d, c \in C \}$ is stable under the action of $G$ and is isomorphic to $\Ind_{\{1,\tau_i\}}^G 1$ as a $G$-module.
\end{proof}

% Let $u_{i,\chi^{-}}:=\sum c^{-1}(u_i) \otimes \chi^{-}(c)$. 
%It follows from the discussion above that $u_{i,\chi^{-}} \neq 0$ and $u_{i,\chi^{-}} \in \cO_H^\times[\chi^{-}]$ for every $1 \leq i \leq d$.
%By Lemma~\ref{lemma: units basis}, we see that $S_i \cap S_j = \emptyset$ for $i \neq j$ and the only algebraic relation between the elements of the set $\cup_{i=1}^{d} S_i$ is given by $\prod_{i=1}^{d}\prod_{c \in C} cu_i = 1$.
%Therefore, it follows that the elements $u_{1,\chi^{-}},\cdots,u_{d,\chi^{-}}$ are linearly independent over $\Qbar$.
%Since $\cO_H^\times[\chi^{-}]$ has dimension $d$, we conclude that $\{u_{i,\chi^{-}}, 1 \leq i \leq d\}$ forms a basis of $\cO_H^\times[\chi^{-}]$.

\begin{defn}\label{def: cmreg}
{Let $\{u_{i,\chi^{-}}, 1 \leq i \leq d\}$ be the basis of $\cO_H^\times[\chi^{-}]$ as in Lemma~\ref{lem: basis}.}
We let the matrix $
\mathrm{R}_p(\chi^{-}):=\left(L_{i,\sigma^{\delta(i)}}(u_{k,\chi^{-}})\right)_{1\leq k,i\leq d}
\subset \mathrm{M}_d(\overline{\Q}_p)$ be the $p$-adic regulator of the anti-cyclotomic character $\chi^{-}$.

\end{defn}

%It follows from Baker-Brumer's theorem that the $p$-adic logarithm induces an injection \[\log_p: \cO^{\times}_H \otimes \bar{\Q} \hookrightarrow \bar\Q_p, u \otimes v \to \log_p(u)  \iota_p(v)\] Thus, $ L_{i,\sigma^{\delta(i)}}(u_{k,\chi^{-}}) \ne 0$ for $1\leq i, k\leq d$. 
%We show in the following Lemma that the $p$-adic Schanuel conjecture predicts the non-vanishing of $\det(\mathrm{R}_p(\chi^{-}))$. 

\begin{lemma}\label{lemma: chichisigmaiso}
Assume that the $p$-adic Schanuel conjecture holds, then the $p$-adic regulator matrix
$\mathrm{R}_p(\chi^{-})$ is invertible.

\end{lemma}

\begin{proof}

See \cite[\S8, p. 3905]{Deo}.

\end{proof}

\begin{lemma}\label{lemma: caracshape}
Let $\theta\in \{0,1\}$ and fix $j\in I_3$. The following are equivalent:
\begin{enumerate}
\item The vector $(b_{i,j})_{1\leq i\leq d}\in\Qbar_p^{d}$ satisfies
\begin{equation}\label{equation: linearsyst} 
\mathrm{R}_p(\chi^{-})\begin{pmatrix}
b_{1,j}\\
\vdots \\
b_{d,j}
\end{pmatrix}=-\theta\begin{pmatrix}
\sum_{c\in C}\chi^{-}(c) L_{j,c\circ \sigma}(u_1))\\
\vdots \\
\sum_{c\in C}\chi^{-}(c) L_{j,c\circ \sigma}(u_d))
\end{pmatrix}\,.
\end{equation}
\item The element
$$
\tilde t_j=\sum_{i\in I_1\cup I_3\setminus \{j\},\,g\in G} 
\varrho_f(g)A_i\varrho_f(g)^{-1}
 L_{i,g} + \sum_{i\in I_2,\,g\in G} 
\varrho_f(g)B_i\varrho_f(g)^{-1}
 L_{i,g}+\sum_{g\in G} 
\varrho_f(g)C_j\varrho_f(g)^{-1}
 L_{j,g}
 $$
of $\big(\Hom \left( (\cO_{H}\otimes \Z_p)^\times, \Qbar_p\right)\otimes \ad^0\varrho_f \big)^G$ belongs to the space $\big(\Hom(G_H,\Qbar_p)\otimes \Ind_K^F\chi^{-}\big)^G$, where
 $$
A_i:=\begin{pmatrix}
0 & b_{i,j}\\
0 & 0
\end{pmatrix},\quad 
B_i:=\begin{pmatrix}
0 & 0\\
b_{i,j} & 0
\end{pmatrix}, \quad 
C_j:=\begin{pmatrix}
0 & b_{j,j}\\
\theta & 0
\end{pmatrix}\,. 
$$

\end{enumerate} 
\end{lemma}

\begin{proof}
Recall that, we have fixed a basis $(v_1,v_2)$ of $\Qbar_p^2$ such that $\varrho_f(c)$ is diagonal for every $c \in C$ and the projective image of $\varrho_f(\sigma)$ is $\begin{pmatrix} 0 & 1\\ 1 & 0\end{pmatrix} \in \PGL_2(\Qbar)$. This gives us the shape of $\tilde t_j$ which implies that $\tilde t_j \in \big(\Hom \left( (\cO_{H}\otimes \Z_p)^\times, \Qbar_p\right)\otimes \Ind_K^F\chi^{-} \big)^G$. Therefore, by \eqref{equation: global-local}, it follows that $\tilde t_j\in \big(\Hom(G_H,\Qbar_p)\otimes \Ind_K^F\chi^{-}\big)^G$ if and only if $\tilde t_j(u)=0$ for every $u\in \cO_H^\times$.

%Recall that, by the Baker-Brumer theorem (see Theorem~\ref{thm: BK}), the map 
%$$
%\cO_{H}^{\times} \otimes \Qbar \rightarrow \Qbar_p\,,\qquad u\otimes v\mapsto \log_p(\iota_p (h_i(u))) \otimes \iota_p(v)
%$$ is injective.
 By definition of $u_1,\dots,u_d$, there exist $y_1,\dots,y_d\in \Qbar$ such that
$$
\sum_{c\in C}\chi^{-}(c)L_{i,c}(u)=\sum_{k=1}^dy_k\sum_{c\in C}\chi^{-}(c)L_{i,c}(u_k)\,
$$

$$
\sum_{c\in C}\chi^{-}(c)L_{i,c\circ \sigma}(u)=\sum_{k=1}^dy_k\sum_{c\in C}\chi^{-}(c)L_{i,c\circ \sigma}(u_k)
$$
for every $1\leq i\leq d$. If $(b_{i,j})_{1\leq i\leq d}\in\Qbar_p^{d}$ satisfies \eqref{equation: linearsyst}, then
\begin{equation}\label{equation: calculot}
\begin{array}{l}
\displaystyle{\sum_{i=1}^d \left( \sum_{c\in C}\chi^{-}(c)L_{i,c\circ\sigma^{\delta(i)}}(u)\right)b_{i,j}=\sum_{k=1}^dy_k\sum_{i=1}^d\left(\sum_{c\in C}\chi^{-}(c)L_{i,c\circ \sigma^{\delta(i)}}(u_k)\right)b_{i,j}=}\\[10pt]
\displaystyle{=-\theta\sum_{k=1}^dy_k\sum_{c\in C}\chi^{-}(c)L_{j,c\circ\sigma}(u_k)=-\theta\sum_{c\in C}\chi^{-}(c)L_{j,c\circ\sigma}(u)}
\end{array}
\end{equation}
Conversely, if \eqref{equation: calculot} is satisfied for every $u\in \cO_H^\times$, then $(b_{i,j})_{1\leq i\leq d}\in\Qbar_p^{d}$ satisfies \eqref{equation: linearsyst}. But it is clear that \eqref{equation: calculot} holds if and only if $\tilde t_j(u)=0$ for every $u\in \cO_H^\times$.
\end{proof}

\begin{cor}
Suppose $\det(\mathrm{R}_p(\chi^{-})) \neq 0$.
For every $j\in I_3$, there exist $b_{i,j}\in \Qbar_p$ for $i\in \{1,\dots, d\}$ such that the elements
$$
\begin{array}{lll}
t_j &:= & \displaystyle{ \sum_{i \in I_1, g\in G} \varrho_f(g) \begin{pmatrix}
0 & b_{i,j} \\
0 & 0
\end{pmatrix}
\varrho_f(g)^{-1} L_{i,g}+
\sum_{i\in I_2, g\in G}\varrho_f(g) 
\begin{pmatrix} 
0 & 0\\
b_{i,j} & 0
\end{pmatrix}
 \varrho_f(g)^{-1} L_{i,g}+}\\[15pt]
 & & \displaystyle{+\sum_{i \in I_3\setminus \{ j\},g\in G} \varrho_f(g) \begin{pmatrix}
 \mathcal{H}_i(b_{i,j},0) & b_{i,j}\\
 0 & -\mathcal{H}_i(b_{i,j},0)
\end{pmatrix} 
  \varrho_f(g)^{-1} L_{i,g}+}\\[15pt]
 & & \displaystyle{+
\sum_{g\in G}\varrho_f(g) 
\begin{pmatrix}
\mathcal{H}_j(b_{j,j},1) & b_{j,j}\\
1 & -\mathcal{H}_j(b_{j,j},1)
\end{pmatrix}
 \varrho_f(g)^{-1} L_{j,g}}\,.

\end{array}
$$
belong to $V$.
\end{cor}

\begin{proof}
For every $j\in I_3$, let $(b_{i,j})_{1\leq i\leq d}$ be a solution of \eqref{equation: linearsyst} with $\theta=1$, which exists because~$\mathrm{R}_p(\chi^{-})$ is invertible. It is clear that the corresponding $t_j\in W$. Note that the antidiagonal component $t^\flat_j$ of $t_j$ is precisely the element $\tilde t_j$ defined in Lemma~\ref{lemma: caracshape}. Therefore, $t_j^\flat \in \big(\Hom(G_H,\Qbar_p)\otimes \Ind_K^F\chi^{-}\big)^G$. Since the isomorphism \eqref{equation: eKisocomp} shows that $t_j^\sharp\in \big(\Hom(G_H,\Qbar_p)\otimes \epsilon_K\big)^G$, we deduce that
$$
t_j=t_j^\sharp+t_j^\flat\in \big(\Hom(G_H,\Qbar_p)\otimes \ad^0\varrho_f\big)^G \cap W = V\,.
$$
\end{proof}

\begin{prop}
\label{dimprop}
Suppose $\det(\mathrm{R}_p(\chi^{-})) \neq 0$.
The elements $(t_j)_{j\in \{1,\dots,d\}}$ constitute a basis for~$V$. In other words, if we let $V_2$ denote the linear span of $(t_j)_{j\in I_3}$, then $V=V_1\oplus V_2$.
\end{prop}

\begin{proof}
Since the elements $t_j$ are linearly independent, we only need to check that they generate~$V$. Let $t$ be any element in $V$. By subtracting suitable multiples of the elements $t_j$ from $t$, we obtain an element $t'$ which in the expression of \eqref{equation: defW} is given by matrices of the form
$$
A_i:=\begin{pmatrix}
0 & b_i\\
0 & 0
\end{pmatrix},\quad 
B_i:=\begin{pmatrix}
0 & 0\\
b_i & 0
\end{pmatrix}, \quad 
C_i:=\begin{pmatrix}
\mathcal{H}_i(b_i,0) & b_i\\
0 & -\mathcal{H}_i(b_i,0)
\end{pmatrix}\,, 
$$
where $b_i\in \Qbar_p$. Since $t'\in V$, we have that
$$
(t')^\sharp\in\big(\Hom(G_H,\Qbar_p)\otimes \Ind_K^F\chi^{-}\big)^G\,. 
$$
Then Lemma~\ref{lemma: caracshape} implies that the vector $(b_i)_{1\leq j\leq d}$ is a solution to \eqref{equation: linearsyst} with $\theta=0$. But the assumption $\det(\mathrm{R}_p(\chi^{-})) \neq 0$ implies that there is no non-zero solution to \eqref{equation: linearsyst} with $\theta=0$. Thus, we have $(t')^\sharp=0$. Since the polynomials $\mathcal{H}_i(X,Y)$ are homogeneous, it follows that each $C_i=0$ and hence, $t'=0$, and we get that $t$ is a linear combination of the $t_j$'s.
\end{proof}

\section{ Proofs of Theorems~\ref{theorem: main1}, \ref{theorem: main2} and Corollary~\ref{theorem: main5}}\label{section: proof of AB}

As explained in \S\ref{section: introduction}, in order to prove Theorem~\ref{theorem: main1} (resp. Theorem~\ref{theorem: main2}), we need to show that 
$$
\dim_{\Qbar_p} t_{\cT'_0}=0\qquad \left(\text{resp. } \dim_{\Qbar_p} t_{\cT}=d+1\right)\,.
$$ 
Let $\cR'_0$ (resp. $\cR$) be the universal deformation ring representing the functor $\cD'_0$ (resp. $\cD$).
Recall that \cite[Prop. 5]{Deo} there is a surjective continuous homomorphism
\begin{equation}\label{equation: surj}
\cR'_0\twoheadrightarrow \cT'_0\qquad \left(\text{resp. }\cR\twoheadrightarrow \cT\right)\,.
\end{equation}
 Combining all the results of \S\ref{section: etaleness}, we get that
$
\dim_{\Qbar_p}t_{\cR'_0}=0
$ under the hypotheses of Theorem~\ref{theorem: main1}.
This means that $\cR'_0 \simeq \Qbar_p$ and hence, $\cR'_0 \simeq \cT'_0 \simeq \Qbar_p$.
This immediately implies Theorem~\ref{theorem: main1}. Note that this also implies $$\cR \simeq \cT$$ under the hypotheses of Theorem~\ref{theorem: main1}.\\
Combining all the results of \S\ref{section: smoothness} we get that $\dim_{\Qbar_p}t_{\cD'} \leq d$ under the hypotheses of Theorem~\ref{theorem: main2}. Combining this inequality with Leopoldt's conjecture for $F$, we get that $\dim_{\Qbar_p} t_{\cR}\leq d+1\,$ (see the proof of \cite[Prop. 1]{Deo} for more details). Therefore, we get surjective maps, $\Qbar_p \lsem T_1,\cdots,T_{d+1} \rsem \twoheadrightarrow \cR \twoheadrightarrow \cT$.
The fact that $\cT$ has Krull dimension $d+1$ implies that the morphism $\Qbar_p \lsem T_1,\cdots,T_{d+1} \rsem \twoheadrightarrow \cT$ obtained by composing the two morphisms above is an isomorphism. Hence, we get that $$\cR \simeq \cT \simeq \Qbar_p \lsem T_1,\cdots,T_{d+1} \rsem.$$ This completes the proof of Theorem~\ref{theorem: main2}.

\subsection*{Proof of Corollary~\ref{theorem: main5}:}\label{padiclfn}

Since every prime $\gp$ of $F$ above $p$ splits in $K$, there exists a nearly-ordinary CM family $\Theta_{\chi}$ specializing to $x$ in weight one. On the other hand, the eigenvariety $\cE$ is smooth at $x$ under our assumption by Theorem~\ref{theorem: main2}. 
%So there exists a unique irreducible component of $\cE $ specializing to $x$. 
Hence, it follows that $\Theta_{\chi}$ is the unique nearly ordinary family passing through $x$. Thus it follows that the congruence ideal $C_0(\chi) \subset \Lambda_0$ attached to the CM family $\Theta_{\chi}$ is not contained in the maximal ideal of $\Lambda_0[1/p]$ given by the equation $X_1=X_2=\cdots = X_{d+1}=0$ and corresponding to $x$  (see \cite[(6.9)]{H-T} for the definition of the congruence ideal $C_0(\chi)$). Let $(H_\chi) \subset \Lambda_0$ be the smallest principal ideal containing the congruence ideal $C_0(\chi) $, so $H_\chi(0) \ne 0$. On the other hand, Hida-Tilouine showed in \cite[Thm.I]{H-T} the divisibility $\zeta_{\chi^{-}}^{-} \mid  H_\chi$ (i.e $(H_\chi) \subset (\zeta_{\chi^{-}}^{-})$). Finally, it follows, from the above divisibility and the fact that $H_\chi(0) \ne 0$, that $\zeta_{\chi^{-}}^{-}(0) \ne 0$.

\end{document}